\documentclass[reqno]{amsart}
\usepackage{amssymb}
\usepackage{amsmath}
\usepackage{amsthm}
\usepackage{mathtools}
\usepackage{graphicx}
\usepackage{comment}

\usepackage{amssymb, mathrsfs, amsfonts, amsmath}

\usepackage{tikz}
\usetikzlibrary{arrows}

\usepackage[english]{babel}

\usepackage[margin=1.4in, centering]{geometry}

\DeclareSymbolFont{bbold}{U}{bbold}{m}{n}
\DeclareSymbolFontAlphabet{\mathbbold}{bbold}
\newtheorem{thm}{Theorem}
\newtheorem{prop}[thm]{Proposition}
\newtheorem{lem}[thm]{Lemma}
\newtheorem{cor}[thm]{Corollary}
\theoremstyle{definition}

\theoremstyle{remark}
\newtheorem{rem}[thm]{Remark}

\newcommand{\BZ}{\mathbb{Z}}

\newcommand{\R}{\mathbb{R}}

\newcommand{\Z}{\mathbb{Z}}

\title{Sparse Bounds for the Bilinear Spherical Maximal Function}
\author{Tainara Borges, Benjamin Foster, Yumeng Ou, Jill Pipher, and Zirui Zhou}

\address[T. Borges]{Department of Mathematics, Brown University, Providence, RI 02912}
\address[B. Foster]{Department of Mathematics, Stanford University, Stanford, CA 94305}
\address[Y. Ou]{Department of Mathematics, University of Pennsylvania, Philadelphia, PA 19104}
\address[J. Pipher]{Department of Mathematics, Brown University, Providence, RI 02912}
\address[Z. Zhou]{Department of Mathematics, University of California, Berkeley, Berkeley, CA 94720}

\thanks{Y.O. is supported by NSF DMS-2055008. \\
B.F. completed part of the work while visiting the Hausdorff Research Institute for Mathematics, funded by the Deutsche Forschungsgemeinschaft (DFG, German Research Foundation) under Germany's Excellence Strategy – EXC-2047/1 – 390685813. }

\begin{document}

\begin{abstract}
We derive sparse bounds for the bilinear spherical maximal function in any dimension $d\geq 1$. When $d\geq 2$, this immediately recovers the sharp $L^p\times L^q\to L^r$ bound of the operator and implies quantitative weighted norm inequalities with respect to bilinear Muckenhoupt weights, which seems to be the first of their kind for the operator. The key innovation is a group of newly developed continuity $L^p$ improving estimates for the single scale bilinear spherical averaging operator. 
\end{abstract}

\maketitle
\section{Introduction}
In this article, we study the bilinear spherical maximal function (applied to nonnegative functions without loss of generality)
\begin{equation}
\mathcal{M}(f,g)(x)=\sup_{t>0}\mathcal{A}_{t}(|f|,|g|)(x),
\end{equation}where at a single scale $t>0$, the averaging operator $\mathcal{A}_t$ is given by the formula
\begin{equation}
    \mathcal{A}_{t}(f,g)(x)=\int_{S^{2d-1}}f(x-ty)g(x-tz)\,d\sigma(y,z).
\end{equation}Here, $d\sigma$ denotes the normalized surface measure on $S^{2d-1}$. We are also interested in the lacunary version of the operator
\begin{equation}
    \mathcal{M}_{lac}(f,g)(x)=\sup_{m\in\BZ}\mathcal{A}_{2^m}(|f|,|g|)(x).
\end{equation}

Our goal is to prove that there is a sparse domination for the trilinear form associated to these operators, i.e.
\begin{equation}
    |\langle \mathcal{M}(f,g),h\rangle| \lesssim \sum_{Q\in \mathcal{S}} |Q| \langle f\rangle_{Q,p}\langle g\rangle_{Q,q}\langle h\rangle_{Q,r}
\end{equation}
(or with $\mathcal{M}_{lac}$ replacing $\mathcal{M}$) for all functions $f,g,h$, where $\mathcal{S}$ is a sparse family and $p, q, r$ satisfy suitable scaling relations. Here, the inner product on the left hand side is the usual $L^2$ inner product and the bracketed quantities on the right side denote $L^p$ averages over the cube $Q$, i.e.
\begin{equation}
    \langle f\rangle_{Q,p}:=\left(\frac{1}{|Q|} \int_Q |f|^p\right)^{1/p}.
\end{equation}
Recall that a collection $\mathcal{S}$ of cubes is called $\eta$-sparse if for every $Q\in \mathcal{S}$, there exists $E_Q\subset Q$ satisfying $|E_Q|\geq \eta |Q|$ and $\{E_Q\}_{Q\in\mathcal{S}}$ are pairwise disjoint. The exact value of the sparse parameter $\eta>0$ in the article is not very important.

Sparse domination has been a rapidly developing area in modern harmonic analysis, since the seminal works of Lerner \cite{Lerner} and Lacey \cite{Lacey1}. It reduces the study of operators of various natures to simpler dyadic operators that are localized and of averaging type. Sparse domination has becoming a leading method in deducing sharp weighted norm inequalities (for operators such as Calder\'on-Zygmund operators \cite{CR, Lacey1, Lerner}, rough singular integrals \cite{CCDPO, DPHL}, the spherical maximal function \cite{Lacey}, Bochner-Riesz multipliers \cite{BBL, LMR}, to name a few, and even non-integral operators \cite{BFP}). In addition to weighted estimates (which in particular include the unweighted $L^p$ space estimates), sparse bounds are also known to imply (often sharp) weak type endpoint estimates. In some sense, the range of exponents for which a sparse bound exists provides refined quantification of the size of the operator. We refer the readers to the aforementioned articles and the references therein for a more detailed account of the history of the sparse domination theory and only highlight below the most relevant prior works to the main result of this article.

One of the most significant developments in the theory was obtained by Lacey \cite{Lacey}, where the sharp (up to the boundary) sparse range for the (linear) spherical maximal function
\[
\mathcal{M}_{linear}(f)(x)=\sup_{t>0}\mathcal{A}_{linear,t}|f|(x):=\sup_{t>0} \int_{S^{d-1}}|f(x-ty)|\,d\sigma(y)
\]and its lacunary analogue were obtained. This operator is highly important in harmonic analysis, especially due to its close connection to the local smoothing estimate for the wave equations \cite{MSS}. Lacey's sparse bound recovered the $L^p$ bounds of $\mathcal{M}_{linear}$ in all dimensions $d\geq 2$, originally due to Stein \cite{Stein} and Bourgain \cite{Bourgain}, and implied novel quantitative weighted norm inequalities. Most importantly, it opened the door to the study of sparse domination for Radon type integral operators and inspired many subsequent works (see e.g. \cite{Oberlin, CO, BC, Hu, CDPPV, BRS, AHRS}). A key achievement of \cite{Lacey} is a framework that reduces the sparse bound for the multi-scale operator to the ``continuity $L^p$ improving'' estimate for the corresponding single scale operators. This general scheme applies to many singular Radon type operators far beyond the spherical maximal function, as demonstrated in the aforementioned subsequent works. In particular, this established a direct connection between the geometry of the manifold being studied and the range of sparse exponents for the operator. 

It would thus be natural to wonder whether similar phenomena also exist in the multilinear setting. Our result is among such efforts in bringing sparse domination into the world of multilinear Radon type operators, where very few results are known. Indeed, even though there is already quite some success in obtaining sparse bounds for multilinear operators such as Calder\'on-Zygmund operators \cite{CR, LN}, rough singular integrals \cite{Barron}, and bilinear Hilbert transforms \cite{CDPO, BM}, the investigation for multilinear Radon transforms only started very recently, before which the weighted norm inequalities for such operators (with respect to multilinear Muckenhoupt weights) were far out of reach. To date, the only known sparse domination results for such operators are due to \cite{Roncal} and \cite{triangle}, where the product type bilinear spherical maximal function and the triangle averaging operator are studied, respectively. The multilinear operators under study here are natural analogues of important linear Radon type operators. There are multilinear Radon type operators which are closely related to questions in continuous geometric combinatorics and geometric measure theory, such as finite point configuration problems and Falconer distance set problems (\cite{multilinearradon,equilateral,discretegeometry}). Our motivation in investigating this natural bilinear spherical maximal function was not in a specific application of this type, but in developing methods that may be useful in future explorations of this general class of bilinear Radon type operators. It will be of interest to understand the 
potential for sparse bounds of such operators as they will imply weighted estimates.

Our main theorem is the following.

\begin{thm}\label{thm: main}
Let $d\geq 2$. For all exponents $(p,q,r)$ such that $r>1$ and $(\frac{1}{p}, \frac{1}{q}, \frac{1}{r})$ is in the boundedness domain
\begin{equation*}
    \mathcal{R}(d)=\left\{
    \left(\frac{1}{p},\frac{1}{q},\frac{1}{r}\right)
    \colon 1< p,q<\infty, 0<r<\infty, \text{ and }\frac{1}{r}<\frac{1}{p}+\frac{1}{q}<m(d,r)\right \}
\end{equation*}
where 
\begin{equation*}
m(d,r)=\begin{cases}
    &\min\{1+\frac{d}{r},\frac{2d-1}{d},\frac{1}{r}+\frac{2(d-1)}{d}\},\text{ if }d\geq 3,\\
    & \min\{1+\frac{1}{r},\frac{3}{2}\}\text{ if }d=2,
    \end{cases}
\end{equation*}the bilinear spherical maximal function $\mathcal{M}$ has a $(p,q,r')$ sparse bound, where $r'$ is given by $\frac{1}{r}+\frac{1}{r'}=1$. More precisely, for all functions $f,g,h \in C_0^\infty$, there exists a sparse collection $\mathcal{S}$ of cubes in $\mathbb{R}^d$ such that
\[
|\langle \mathcal{M}(f,g),h\rangle| \lesssim \sum_{Q\in \mathcal{S}} |Q| \langle f\rangle_{Q,p}\langle g\rangle_{Q,q}\langle h\rangle_{Q,r'}.
\]
\end{thm}

The $L^p\times L^q \to L^r$ boundedness of the operator $\mathcal{M}$ has attracted a great amount of attention recently and has been studied by multiple authors (see e.g. \cite{GGIPS, BGHHO, GHH, HHY, JL}). In \cite{JL}, Jeong and Lee studied the operator in $d\geq 2$ via its slices and made a breakthrough observation that there is a pointwise bound
\begin{equation}\label{eqn: JL pointwise}
\mathcal{M}(f,g)(x)\lesssim M(f)(x)\mathcal{M}_{linear}(g)(x),
\end{equation}where $M$ denotes the Hardy-Littlewood maximal function. Their slicing approach that led to this inequality involved using the coarea formula to decompose an integral over $S^{2d-1}$ into an iterated integral over the $d$-dimensional unit ball and $(d-1)$-dimensional spheres. The pointwise domination beautifully concludes the study of the Lebesgue space bounds for $\mathcal{M}$. More precisely, (\ref{eqn: JL pointwise}) implies the sharp boundedness of $\mathcal{M}$ (\cite[Theorem 1.1]{JL}): for $d\geq 2$, $1\leq p,q\leq\infty$, $0<r\leq \infty$, $\mathcal{M}$ maps boundedly from $L^p\times L^q$ to $L^r$ if and only if $\frac{1}{p}+\frac{1}{q}=\frac{1}{r}$ and $r>\frac{d}{2d-1}$ except the case $(p,q,r)=(1,\infty,1)$ or $(\infty,1,1)$. Unfortunately, (\ref{eqn: JL pointwise}) doesn't imply that the bilinear sparse bound for $\mathcal{M}$ will be inherited from the linear ones for $M$ and $\mathcal{M}_{linear}$. Indeed, it is not generally true that the pointwise product of two sparse operators is dominated by a bilinear sparse operator. Nevertheless, our proof relies heavily on the slicing strategy.

\begin{rem}
Since the lacunary maximal function $\mathcal{M}_{lac}$ is obviously dominated by $\mathcal{M}$, the same result in Theorem \ref{thm: main} also extends to $\mathcal{M}_{lac}$. However, it seems that our method doesn't imply an improved result (i.e. larger sparse range) for $\mathcal{M}_{lac}$, unlike the linear case in \cite{Lacey} or the case of the product type bilinear spherical maximal function \cite{Roncal}. This is because our proof involves slicing the bilinear operator, which would result in a linear single scale averaging operator for a continuum spectrum of radii. The gain brought by the lacunary scales would thus vanish. In fact, to the best of our knowledge, the optimal range of even the Lebesgue space bounds $L^p\times L^q\to L^r$ for $\mathcal{M}_{lac}$ is still unknown, other than the $d=1$ case very recently obtained in \cite{MCZZ}.
\end{rem}

Using the sparse bounds obtained in Theorem \ref{thm: main} and standard arguments, one can immediately recover the full range of the $L^{p_0}\times L^{q_0} \to L^{r_0}$ boundedness of $\mathcal{M}$ proved in \cite[Theorem 1.1]{JL}. Indeed, in general, let $T$ be a bi-sublinear operator. Then the $(p,q,r')$ sparse bound for $\langle T(f,g),h\rangle$ implies the $L^{p_0}\times L^{q_0}\to L^{r_0}$ bound of $T$, for all $p_0\in (p,\infty]$, $q_0\in (q,\infty]$ with $\min(p_0,q_0)<\infty$, and $\frac{1}{r_0}=\frac{1}{p_0}+\frac{1}{q_0}$ (for instance, see \cite[Proposition 1.2]{CDPO}). In Theorem \ref{thm: main}, if one chooses $r$ so that $m(d,r)=\frac{2d-1}{d}$ (for instance take any $r\in (1,d)$) and chooses $p,q$ so that $\frac{1}{p}+\frac{1}{q}$ is close to $\frac{2d-1}{d}$, the desired $L^{p_0}\times L^{q_0}\to L^{r_0}$ norm bounds then follow. It is also of interest to study the mapping properties of $\mathcal{M}$ on Lorentz spaces (see e.g. \cite[Theorem 1.1]{JL}) and it seems likely that the method developed here could also lead to a sparse domination for $\mathcal{M}$ that involves Lorentz norm averages. It would be interesting to develop such versions of the sparse domination and see whether they imply improved Lorentz space estimates. We leave that for future exploration.

Moreover, via standard arguments, Theorem \ref{thm: main} implies for the first time quantitative weighted norm inequalities for $\mathcal{M}$ with respect to multilinear Muckenhoupt weights. Since this deduction is routine and is not the focus of the article, we refer the reader to \cite{Roncal, triangle} and the references therein for a more comprehensive discussion on the weighted corollaries. It is unknown though whether the weighted norm inequalities derived are sharp.

In Theorem \ref{thm: main}, some of the sufficient conditions on the range of exponents $(p,q,r)$ for which sparse bounds hold true are also necessary. We will show this in Section \ref{sec: sharp} using several explicitly constructed examples. Some of the examples already appeared in \cite{JL} in connection to their $L^p$ improving estimate for the localized bilinear maximal function $\tilde{\mathcal{M}}$ (see (\ref{eqn: loc M}) for its definition). We also construct a new example that, in addition to providing necessary conditions for the sparse bounds, gives a previously unknown necessary condition for such $L^p$ improving estimate,  strengthening \cite[Proposition 3.3]{JL}. It is unknown though whether this new necessary condition is sufficient for the $L^p$ improving estimate of $\tilde{\mathcal{M}}$.

In dimension $d=1$, we also obtain continuity $L^p$ improving estimates for the single scale averaging operator and hence sparse bounds for the bilinear spherical maximal function, but only for the lacunary version (Theorem \ref{thm: d1} in Section \ref{sec: d1}). In general, the lower the ambient dimension is, the more singular the spherical maximal function becomes, which can be seen as a consequence of the slower decay of the Fourier transform of the surface measure on the sphere. In particular, in $d=1$, the bilinear circular maximal function behaves very differently from all its higher dimensional analogues. For instance, the boundedness result for $\mathcal{M}$ obtained in \cite{JL} only applies in $d\geq 2$, and such Lebesgue space bounds in $d=1$ were only very recently obtained, independently, in \cite{MCZZ, DR} via somewhat similar approaches involving parametrizing pieces of the circle. For the single scale operators, the study of the $L^p$ improving estimates for $\mathcal{A}_t$ in $d=1$ was initiated by D. Oberlin \cite{DOberlin} several decades ago but one still only has limited knowledge on the optimal range for such estimates (see \cite{BS, SS}). When it comes to the localized maximal function $\tilde{\mathcal{M}}$ defined in (\ref{eqn: loc M}) below, as far as the authors are aware, there is no known result on its boundedness outside the H\"older range ($\frac{1}{p}+\frac{1}{q}=\frac{1}{r}$) in the literature. In fact, our results in the $d=1$ case (on the sparse bound and the continuity estimate for the single scale operators) seem to be the first of their kinds concerning Radon type operators in the bilinear setting in dimension $d=1$. Indeed, both the previously studied operators in this direction, the product type bilinear spherical maximal function \cite{Roncal} and the triangle maximal operator \cite{triangle}, are only defined in $d\geq 2$.

\subsection*{Novelty of the proof}
As in the linear case studied in \cite{Lacey}, the core of the matter here is to obtain a continuity $L^p$ improving estimate for the single scale operator, which, in the case of the full version of the bilinear spherical maximal function $\mathcal{M}$, means the following restricted range maximal function
\begin{equation}\label{eqn: loc M}
\tilde{\mathcal{M}}(f,g)(x):=\sup_{t\in [1,2]}|\mathcal{A}_t(f,g)(x)|.
\end{equation}Let $d\geq 2$. It is shown in \cite{JL} that for any tuple $(p,q,r)$ such that $(\frac{1}{p}, \frac{1}{q}, \frac{1}{r}) \in \mathcal{R}(d)$ (or certain parts of the boundary of the region), $\tilde{\mathcal{M}}$ maps from $L^p\times L^q$ into $L^r$, which is often referred to as the $L^p$ improving estimate, as the target space $L^r$ is better than the space corresponding to the H\"older exponent. Following closely the methods in \cite{Roncal, triangle}, one can reduce the desired sparse bounds to the continuity $L^p$ improving estimates of the form
\[
\|\sup_{1\leq t\leq 2}|\mathcal{A}_t(f,g)-\mathcal{A}_t(f,\tau_h g)|\|_{L^r}\leq C |h|^\eta \|f\|_{L^p}\|g\|_{L^q},
\]where $|h|<1$, $\eta=\eta(d,p,q,r)>0$ and $\tau_h g(\cdot)=g(\cdot-h)$ denotes the translation operator by $h$.

Compared to earlier works, we are faced with several unique challenges here. First, our operator doesn't have a product structure, making it difficult to deduce multilinear estimates from their linear counterparts. This is in contrast to the situation in \cite{Roncal}, where the product type bilinear spherical maximal function 
\[
\mathcal{M}_{prod}(f,g)(x):=\sup_{t>0}\mathcal{A}_{linear, t}(f)\mathcal{A}_{linear, t}(g)(x)
\]is studied. For this operator, the continuity estimate obtained for the linear averaging operator $\mathcal{A}_{linear, t}$ in \cite{Lacey} easily implies the desired bilinear bounds. 

Second, even though it is possible to use the slicing technique to reduce to the linear averaging operator, the slices obtained for our operator are not local. For example, considering the case $t=1$, a simple calculation \cite[proof of Lemma 2.1]{JL} shows that
\[
\begin{split}
&|\mathcal{A}_1(f,g)(x)|\\
\leq &\int_{B^d(0,1)}|f(x-y)|\left|\int_{S^{d-1}}g(x-\sqrt{1-|y|^2}z)\,d\sigma_{d-1}(z)\right|(1-|y|^2)^{\frac{d-2}{2}}\,dy.
\end{split}
\]For each $y$ fixed, the inner integral is the linear spherical averaging of $g$ at radius $\sqrt{1-|y|^2}$. One may thus be tempted to insert the linear continuity estimate from \cite{Lacey} here to estimate the inner integral. However, with varying $y$, the radius $\sqrt{1-|y|^2}$ may become too small (smaller than the translation parameter $|h|$) for the continuity estimate to apply. This is in contrast to the triangle maximal operator studied in \cite{triangle}
\[
T(f,g)(x):=\sup_{t>0}\left|\int_{\mathcal{Z}}f(x-ty)g(x-tz)\,d\mu(y,z)\right|,
\]where $\mathcal{Z}=\{(y,z)\in \mathbb{R}^{2d}:\, |y|=|z|=|y-z|=1\}$ and $\mu$ is the natural surface measure on $\mathcal{Z}$ as an embedded submanifold of $\mathbb{R}^{2d}$. Notice that conditions $|y|=|z|=1$ ensure that this operator is always localized at scale $t$ for each input function, making it possible to deduce the bilinear continuity estimate fairly straightforwardly from the linear one.

In fact, one may actually try to prove the sparse bound in Theorem \ref{thm: main} solely based on slicing and the continuity estimate for the linear single scale averaging operator, bypassing the bilinear continuity estimate completely. However, the same difficulty seems to show up: the non-local feature of the operator would introduce extra singularities for extreme values of $y$.

The main novelty of Theorem \ref{thm: main} lies in how we overcome these barriers and obtain the desired bilinear continuity estimate. When $d\geq 3$, we apply an intermediate decomposition in \cite{JL} to localize the operator to different dyadic scales and dominate each piece by a quantity involving the linear version of the localized spherical maximal operator. The desired estimate then follows from slicing and the continuity estimate for the linear operator obtained by Lacey \cite{Lacey}. 

Things become much more complicated when $d=2$, where the dyadic decomposition fails to generate a convergent sum. This seems to manifest the fact that the spherical maximal function, even in the linear case, behaves more wildly in $d=2$, where the $L^2$ based argument stops working (which was why Bourgain's proof for the $L^p$ bound of $\mathcal{M}_{linear}$ in $d=2$  \cite{Bourgain} came much later than Stein's proof for $d\geq 3$ \cite{Stein}). In this case, the earlier argument relying on slicing and reducing to the continuity estimate for the linear localized spherical maximal function doesn't work any more. Instead, we conduct a direct and more careful analysis for the bilinear operator. The analysis is much more delicate and several new ingredients come into play. For instance, one needs to first break up the operator into two parts and prove a continuity estimate for each part at different suitable tuples of exponents. Interpolation then allows one to conclude the continuity estimate for the original operator. Moreover, a change of variable reduces the problem to estimating the difference between averaging operators with different radii, which is studied via a further decomposition into dyadic scales. A key estimate in this part is Lemma \ref{lemma2d}, which is a rescaled version of \cite[Proposition 4.2]{Lacey}. 

In the case $d=1$, in order to obtain sparse bounds for the lacunary bilinear spherical maximal function $\mathcal{M}_{lac}$, one needs to study the continuity estimate for the single scale averaging operator $\mathcal{A}_t$, which we summarize in Theorem \ref{d1continuity} in Section \ref{sec: d1}. The proof of this theorem follows a completely different approach, as one is unable to slice the bilinear operator in $d=1$. Our proof directly makes use of information in the frequency space and relies on a very recently developed trilinear smoothing inequality \cite{MCZZ}, which implies Sobolev norm bounds for the averaging operator. 

In addition, in all dimensions $d\geq 1$, by making use of the shape of the region of exponents $\mathcal{R}(d)$, we manage to prove the sparse bounds for the tuple $(p,q,r)$ without assuming that $p, q\leq r$. This is in contrast to prior works \cite{Roncal, triangle}, where (certain versions of) this constraint showed up. This will be explained towards the end of the proof of Theorem \ref{thm: main} in Section \ref{sec: sparse lemma} where the ``Bad-Good'' case is discussed.

\subsection*{Plan of the article.} 
We first introduce the bilinear continuity $L^p$ improving estimates for $\tilde{\mathcal{M}}$ in $d\geq 2$ in Section \ref{sec: continuity}, and then move onto the $d=1$ case in Section \ref{sec: d1} where the continuity estimate for the single scale operator $\mathcal{A}_t$ is obtained. These are the key intermediate results of the article and may be of independent interest. We then give an initial reduction to dyadic maximal operators in Section \ref{sec: reduction}, and present the main lines of the proof of Theorem \ref{thm: main} in Section \ref{sec: sparse lemma}. Both of these two sections follow closely the arguments in \cite{Roncal, triangle} and are included mainly for the sake of completeness and to explain why the constraint $p,q\leq r$ can be removed. In Section \ref{sec: sharp}, we give examples to show sharpness of part of the range of exponents for the sparse bounds and the continuity $L^p$ improving estimates.

\begin{rem}
During the completion of the article, the authors learned that an updated version of \cite{triangle}, simultaneously under completion, includes an elegant expanded framework that covers general multilinear Fourier multipliers which in particular includes the bilinear spherical maximal function as an example. The proof in \cite{triangle} (for the key continuity estimate) is of a different nature, which beautifully links the problem to the decay of the Fourier transform of the surface measure on the sphere. However, their result only holds in dimension $d\geq 4$. This is probably not surprising, as the Fourier decay of the surface measure becomes worse in low dimensions, where the behavior of the spherical maximal function and Radon type operators is in general more delicate.
\end{rem}

\section{Continuity estimates in $d\geq 2$}\label{sec: continuity}
In this section, we obtain several continuity estimates for the localized bilinear spherical maximal function, which will play a key role in the proof of Theorem \ref{thm: main}. Recall that we have defined
\[\mathcal{A}_{t}(f,g)(x)=\int_{S^{2d-1}} f(x-ty)g(x-tz)d\sigma(y,z).
\]Recall also the region
\begin{equation} \label{eq: bdd region}
    \mathcal{R}(d)=\left\{
    \left(\frac{1}{p},\frac{1}{q},\frac{1}{r}\right)
    \colon 1< p,q<\infty, 0<r<\infty, \text{ and }\frac{1}{r}<\frac{1}{p}+\frac{1}{q}<m(d,r)\right \}
\end{equation}
where 
\begin{equation} \label{eq: upper bd pieces}
m(d,r)=\begin{cases}
    &\min\{1+\frac{d}{r},\frac{2d-1}{d},\frac{1}{r}+\frac{2(d-1)}{d}\},\text{ if }d\geq 3.\\
    & \min\{1+\frac{1}{r},\frac{3}{2}\}\text{ if }d=2.
    \end{cases}
\end{equation}

The main result of this section is the following.
\begin{prop}\label{prop: cont}
Let $d\geq 2$, and $(\frac{1}{p},\frac{1}{q},\frac{1}{r}) \in\mathcal{R}(d)$. There exists $\eta=\eta(p,q,r,d)>0$ such that
\[
\left\|\sup_{s\leq t\leq 2s} |\mathcal{A}_t(f, g-\tau_h g)|\right\|_{L^r}\leq C s^{d(\frac{1}{r}-\frac{1}{p}-\frac{1}{q})}\left( \frac{|h|}{s}\right)^\eta \|f\|_{L^p}\|g\|_{L^q}
\]for all $s>0$ and $|h|<s$.

Moreover, there exist positive constants $\eta_i=\eta_i(d,p,q,r)$, $i=1,2$, such that 
$$\left\|\sup_{s\leq t\leq 2s}|\mathcal{A}_t(f-\tau_{h_1}f,g-\tau_{h_2}g)|\right\|_{L^r}\leq C s^{d(\frac{1}{r}-\frac{1}{p}-\frac{1}{q})}\left( \frac{|h_1|}{s}\right)^{\eta_1}\left( \frac{|h_2|}{s}\right)^{\eta_2}\|f\|_{L^p}\|g\|_{L^q}$$
for all $s>0$ and $|h_1|<s,\,|h_2|<s$.
\end{prop}

By rescaling, it suffices to consider the case $s=1$. In the rest of the section, we will thus consider only the localized bilinear spherical maximal function
$$\tilde{\mathcal{M}}(f,g)(x):=\sup_{1\leq t\leq 2} |\mathcal{A}_t(f,g)|.$$

We first consider the case $d\geq 3$, which is simpler.

\begin{prop}\label{continuity} ($\eta$ regularity condition for the bilinear spherical maximal operator for $d\geq 3$) Let $d\geq 3$, $1\leq p,q\leq \infty$, $0<r< \infty$. One has 
\begin{equation}\label{hestimate}
    \left\|\sup_{1\leq t\leq 2}|\mathcal{A}_t(f,g-\tau_h g)|\right\|_{L^r}\leq C |h|^{\eta}\|f\|_{L^p}\|g\|_{L^q}
\end{equation}
for all $|h|<1$, where $\eta=\eta(d,p,q,r)>0$, provided that 
     \[
     r>\dfrac{d}{2d-1} \text{ and }\, \left(\frac{1}{p},\frac{1}{q},\frac{1}{r}\right)\in \mathcal{R}(d),
     \]
where $\mathcal{R}(d)$ is given in \eqref{eq: bdd region}.
\end{prop}

\begin{rem} Note that defining 
$$m=\min\left\{1+\dfrac{d}{r}, \dfrac{2d-1}{d}, \dfrac{1}{r}+\dfrac{2(d-1)}{d}\right\},$$ then 

\begin{itemize}
    \item $m=\dfrac{2d-1}{d}$, \, if $\dfrac{1}{d}< \dfrac{1}{r}< \dfrac{2d-1}{d}$;
    \item $m=\dfrac{1}{r}+\dfrac{2(d-1)}{d}$, \, if $\dfrac{d-2}{d(d-1)}\leq\dfrac{1}{r}\leq \dfrac{1}{d}$;
    \item $m=1+\dfrac{d}{r}$, \, if $\dfrac{1}{r}< \dfrac{d-2}{d(d-1)}$.
\end{itemize}
\end{rem}

\begin{proof}[Proof of Proposition \ref{continuity}]

Note that the bound follows from the triangle inequality when $|h|\ge 1$, so we may assume $|h|$ is small. Given any measurable function $\kappa:\R^d\rightarrow [1,2]$, define
$$\tilde{\mathcal{A}}_{\kappa}(f,g)(x)=\int_{S^{2d-1}}f(x-\kappa(x)y)g(x-\kappa(x)z)d\sigma(y,z).$$

By linearization, all we need to show is that 
\begin{equation}\label{goal}
   \|\tilde{\mathcal{A}}_{\kappa}(f,g-\tau_{h}g)\|_r\lesssim |h|^{\eta}\|f\|_p\|g\|_q 
\end{equation}
for $(p,q,r)$ as in the proposition, with $\eta$ and the implicit constant independent of $\kappa$. 

    Recall the following decomposition from \cite[Proposition 3.2]{JL}:
    \begin{equation}\label{split}
        \tilde{\mathcal{A}}_{\kappa}(f,g)(x)\lesssim |f|*\chi_B(x) \left(\sum_{l=0}^{\infty} 2^{-l(\frac{d-2}{2})}\tilde{\mathcal{S}}[g(2^{-l/2}\cdot)](2^{l/2}x)\right), 
    \end{equation}
    where $B=B^{d}(0,2)$ and 
    \begin{equation} \label{loc sphere}
        \tilde{\mathcal{S}}f(x)=\sup_{1\leq t\leq 2}| \int_{S^{d-1}}f(x-ty)d\sigma(y)|
    \end{equation}
    denotes the localized linear spherical maximal function. This estimate is a consequence of slicing at each fixed $y$ and decomposing the resulting linear spherical average of $g$ into different dyadic scales.

    Hence, replacing $g$ by $g-\tau_h g$, one gets
     \begin{equation}\label{split2}
        \tilde{\mathcal{A}}_{\kappa}(f,g-\tau_h g)(x)\lesssim |f|*\chi_B(x) \left(\sum_{l=0}^{\infty} 2^{-l(\frac{d-2}{2})}\tilde{\mathcal{S}}[(g-\tau_h g)(2^{-l/2}\cdot)](2^{l/2}x)\right). 
    \end{equation}
    
    Following the notation of \cite{JL}, for $d\geq 2$, let $\Delta(d)$ denote the closed region which is the convex hull of the vertices
    \begin{equation}\label{vertices}
        \mathcal{B}_{1}^d=(0,0),\,\mathcal{B}_{2}^d=\left(\frac{d-1}{d},\frac{d-1}{d}\right),\, \mathcal{B}_{3}^d=\left(\frac{d-1}{d},\frac{1}{d}\right) \text{ and }\mathcal{B}_{4}^d=\left(\frac{d^2-d}{d^2+1},\frac{d-1}{d^2+1}\right)
    \end{equation}

    When $d=2$, $\mathcal{B}_2^2=\mathcal{B}_{3}^{2}=(1/2,1/2)$ so this region is actually a triangle. The relevance of the region $\Delta(d)$ is closely related to the $L^p$ improving region for $\tilde{\mathcal{S}}$ \cite{Lee, Schlag,SchlagSogge}, as we recall in the following proposition.

\begin{prop}\label{linearlpimproving}\cite[Theorem 1.1, Theorem 1.4]{Lee} Let $d\geq 3$ and $(1/q,1/r)\in \Delta(d)\backslash \{\mathcal{B}_2^d,\mathcal{B}_3^d,\mathcal{B}_4^d\}$, or $d=2$ and $(1/q,1/r)\in\Delta(2)\backslash \{(1/2,1/2),(2/5,1/5)\}$, then
\begin{equation}
    \|\tilde{\mathcal{S}}f\|_{r}\leq C\|f\|_q.
\end{equation}

\end{prop}
From (\ref{split2}) and the $L^{p}$ improving estimates of $\tilde{\mathcal{S}}$ from Proposition \ref{linearlpimproving} we will get the following lemma, which is a version of Proposition \ref{continuity} for a more restricted set of exponents $(p,q,r)$:
 \begin{lem}\label{hlemma}  Let $d\geq 3$, $1\leq p\leq \infty$ and $(1/q,1/r)\in \textnormal{int}(\Delta(d))$ such that 
 $$\dfrac{1}{r}>\dfrac{1}{q}-\dfrac{d-2}{d}. $$
 Then there exists $\eta=\eta(d,p,q,r)>0$ such that
\begin{equation}\label{hestimate2}
   \|\tilde{\mathcal{A}}_{\kappa}(f,g-\tau_h g)\|_r\lesssim |h|^{\eta}\|f\|_p\|g\|_q 
\end{equation}
for all $|h|<1$.
 \end{lem}
 
 \begin{figure}[h]
 \begin{center}
\scalebox{1}{
\begin{tikzpicture}

\draw (-0.3,3.8) node {$\frac{1}{r}$};
\draw (3.8,0) node {$\frac{1}{q}$};
\draw (3,-0.4) node {$1$};
\draw (2,-0.5) node{$\frac{d-1}{d}$};
\draw (-0.4,3) node {$1$};
\draw (-0.5,2) node{$\frac{d-1}{d}$};
\draw (-0.4,1) node {$\frac{1}{d}$};
\draw (1,-0.5) node {$\frac{d-2}{d}$};
\draw[red] (-0.6,0.5) node {$\frac{d-2}{d(d-1)}$};
\draw[red] (1.5,-1) node {$\frac{d-2}{d-1}$};
\draw (-0.3,-0.3) node{\small{$\mathcal{B}_{1}^d$}};
\draw (2.3,2.3) node{\small{$\mathcal{B}_{2}^d$}};
\draw (2.3,1) node{\small{$\mathcal{B}_{3}^d$}};
\draw (2.1,0.5) node {\small{$\mathcal{B}_{4}^d$}};

\draw[->,line width=1pt] (-0.2,0)--(3.5,0);
\draw[->,line width=1pt] (0,-0.2)--(0,3.5);
\draw[-, dashed, line width=1pt,gray] (2,0)--(2,3);

\draw[-,line width=1pt] (2,-0.1)--(2,0.1);
\draw[-,line width=1pt] (3,-0.1)--(3,0.1);
\draw[-,line width=1pt] (-0.1,2)--(0.1,2);
\draw[-,line width=1pt] (-0.1,3)--(0.1,3);
\draw[-,line width=1pt] (-0.1,1)--(0.1,1);
\draw[-,line width=1pt] (1,-0.1)--(1,0.1);
\draw[-,line width=1pt,red] (-0.1,0.5)--(0.1,0.5);
\draw[-,line width=1pt,red] (1.5,-0.1)--(1.5,0.1);

\draw[-,dashed, line width=0.8pt] (0,0)--(2,2);
\draw[-,dashed,line width=0.8pt, red] (1.5,-0.7)--(1.5,0.5);

\draw[-, dashed, line width=0.8pt] (2,1)--(2,2);
\draw[-, dashed, line width=0.8pt] (2,1)--(9/5,3/5);
\draw[-, dashed, line width=1pt] (0,0)--(9/5,3/5);

\fill[red!30!white] (0,0)--(1.5,0.5)--(2,1)--(2,2)--(0,0);
\draw[-,line width=1pt,red] (1,0)--(3,2);
\draw[red] (3.9,2.3) node {$\frac{1}{r}=\frac{1}{q}-\frac{d-2}{d}$};

\filldraw[black] (0,0) circle (1.5pt);
\filldraw[black] (2,2) circle (1.5pt);
\filldraw[black] (2,1) circle (1.5pt);
\filldraw[black] (9/5,3/5) circle (1.5pt);
\filldraw[red] (1.5,0.5) circle (1.5pt);

\draw[-,dashed,line width=0.8pt, red] (0,0.5)--(1.5,0.5);
\end{tikzpicture}}
\end{center}
\caption{The red region in the figure above illustrates the region where the pairs $(1/q,1/r)$ in Lemma \ref{hlemma} live.}
\label{fig1}
\end{figure}
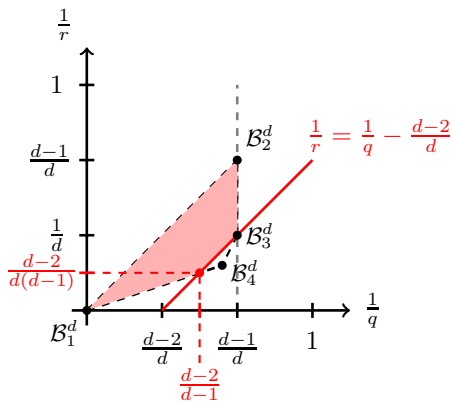

We will defer the proof of Lemma $\ref{hlemma}$ for now and instead see how it implies Proposition \ref{continuity}. By the boundedness properties of $\tilde{\mathcal{A}}_{\kappa}$ given in \cite[Proposition 3.2]{JL},
\begin{equation}\label{nohgain}
    \begin{split}
        \|\tilde{\mathcal{A}}_{\kappa}(f,g-\tau_h g)\|_r&\lesssim \|\tilde{\mathcal{A}}_{\kappa}(f,g)\|_r+\|\tilde{\mathcal{A}}_{\kappa}(f,\tau_h g)\|_r\\
        &\lesssim \|f\|_p\|g\|_q
    \end{split}
\end{equation}
as long as $1/r\leq 1/p+1/q< m$, where $m=m(d,r)$ as in \eqref{eq: upper bd pieces}.

Lemma \ref{hlemma} gives us an open region inside the half-space $\{r>d/(d-1)\}$ where the desired continuity estimate holds. We can use bilinear interpolation with the estimate above (which has no $|h|$ gain) and we will get continuity estimates of the form 
$$\|\tilde{\mathcal{A}}_{\kappa}(f,g-\tau_h g)\|_r\lesssim |h|^{\eta} \|f\|_p\|g\|_q,$$
in the whole open region $1/r< 1/p+1/q< m$.

Inspired by Figure \ref{fig1} we see that it is interesting to subdivide the points that one gets from Lemma \ref{hlemma} in three sub-regions, according to the value of $1/r$, to have a more clear description of the points $(1/p,1/q,1/r)$ that one is getting from that lemma.
In each of the following figures, the red region represents the continuity region given by Lemma \ref{hlemma}, and the dashed contour in blue is the boundary of the open region $1/r<1/p+1/q<m(d,r)$, for which we will get the proposition via interpolation with (\ref{nohgain}).

\begin{enumerate}
    \item \textit{Case 1:}$\quad\frac{1}{d}<\frac{1}{r}<\frac{d-1}{d}$.
    
    In this case $(\frac{1}{q}, \frac{1}{r})\in \text{int}(\Delta(d))$ with $ \frac{1}{r}>\frac{1}{q}-\frac{d-2}{d} $  if and only if $\frac{1}{r}< \frac{1}{q}<\frac{d-1}{d}$.
    
    \begin{center}
         \scalebox{0.9}{
\begin{tikzpicture}

\fill[blue!10!white] (0,1.5)--(1.5,0)--(3,0)--(3,2)--(2,3)--(0,3)--(0,1.5);
\fill[red!30!white] (0,1.5)--(0,2)--(3,2)--(3,1.5)--(0,1.5);
\draw[red,dashed,line width=1pt](0,1.5)--(0,2)--(3,2)--(3,1.5)--(0,1.5);

\draw (-0.3,3.8) node {$\frac{1}{q}$};
\draw (3.8,0) node {$\frac{1}{p}$};
\draw (3,-0.4) node {$1$};
\draw (-0.4,3) node {$1$};
\draw (-0.4,1) node {$\frac{1}{d}$};
\draw (1,-0.5) node {$\frac{1}{d}$};
\draw (1.5,-0.4) node {$\frac{1}{r}$};
\draw[red] (-0.4,1.5) node {$\frac{1}{r}$};
\draw[red] (-0.5,2) node{$\frac{d-1}{d}$};
\draw (2,-0.4) node{$\frac{d-1}{d}$};

\draw[->,line width=1pt] (-0.2,0)--(3.5,0);
\draw[->,line width=1pt] (0,-0.2)--(0,3.5);

\draw[-,line width=1pt] (3,-0.1)--(3,0.1);
\draw[-,line width=1pt] (-0.1,3)--(0.1,3);
\draw[-,line width=1pt] (-0.1,2)--(0.1,2);
\draw[-,line width=1pt] (2,-0.1)--(2,0.1);
\draw[-,line width=1pt] (-0.1,1)--(0.1,1);
\draw[-,line width=1pt] (1,-0.1)--(1,0.1);

\draw[-,line width=1pt] (-0.1,1.5)--(0.1,1.5);
\draw[-,line width=1pt] (1.5,-0.1)--(1.5,0.1);

\filldraw[black] (0,0) circle (1.5pt);

\draw[dashed,blue,line width=1pt] (0,1.5)--(1.5,0)--(3,0)--(3,2)--(2,3)--(0,3)--(0,1.5);

\draw (3.7,2.7) node {\small{$\frac{1}{p}+\frac{1}{q}=\frac{2d-1}{d}$}};

\end{tikzpicture}}
\end{center}
    
    \item \textit{Case 2:} $\quad\frac{d-2}{d(d-1)}\leq \frac{1}{r}\leq \frac{1}{d}$.
    
   In this case, $(\frac{1}{q}, \frac{1}{r})\in \text{int}(\Delta(d))$ with $ \frac{1}{r}>\frac{1}{q}-\frac{d-2}{d} $  if and only if $\frac{1}{r}<\frac{1}{q}<\frac{1}{r}+\frac{d-2}{d}$.
    
       \begin{center}
         \scalebox{0.9}{
\begin{tikzpicture}

\fill[blue!10!white] (0,0.75)--(0.75,0)--(3,0)--(3,1.75)--(1.75,3)--(0,3)--(0,0.75);
\fill[red!30!white] (0,0.75)--(0,1.75)--(3,1.75)--(3,0.75)--(0,0.75);
\draw[red,dashed,line width=1pt](0,0.75)--(0,1.75)--(3,1.75)--(3,0.75)--(0,0.75);

\draw (-0.3,3.8) node {$\frac{1}{q}$};
\draw (3.8,0) node {$\frac{1}{p}$};
\draw (3,-0.4) node {$1$};
\draw (-0.4,3) node {$1$};
\draw (-0.4,1) node {$\frac{1}{d}$};
\draw (1,-0.4) node {$\frac{1}{d}$};
\draw (0.75,-0.4) node {$\frac{1}{r}$};
\draw[red] (-2,0.75) node {$\frac{1}{r}$};
\draw (-0.4,2) node{$\frac{d-1}{d}$};
\draw (2,-0.4) node{$\frac{d-1}{d}$};
\draw[red] (-2,1.75) node {$\frac{1}{r}+\frac{d-2}{d}$};
\draw(-0.6,0.5) node {\small{$\frac{d-2}{d(d-1)}$}};

\draw[->,line width=1pt] (-0.2,0)--(3.5,0);
\draw[->,line width=1pt] (0,-0.2)--(0,3.5);
\draw[-,line width=1pt, red, dashed] (-1.2,1.75)--(3.4,1.75);
\draw[-,line width=1pt, red, dashed] (-1.2,0.75)--(3.4,0.75);

\draw[-,line width=1pt] (3,-0.1)--(3,0.1);
\draw[-,line width=1pt] (-0.1,3)--(0.1,3);
\draw[-,line width=1pt] (-0.1,2)--(0.1,2);
\draw[-,line width=1pt] (2,-0.1)--(2,0.1);
\draw[-,line width=1pt] (-0.1,1)--(0.1,1);
\draw[-,line width=1pt] (1,-0.1)--(1,0.1);
\draw[-,line width=1pt] (-0.1,0.75)--(0.1,0.75);
\draw[-,line width=1pt] (0.75,-0.1)--(0.75,0.1);
\draw[-,line width=1pt] (-0.1,1.75)--(0.1,1.75);
\draw[-,line width=1pt] (1.75,-0.1)--(1.75,0.1);
\draw[-,line width=1pt] (-0.1,0.5)--(0.1,0.5);

\filldraw[black] (0,0) circle (1.5pt);

\draw[dashed,blue,line width=1pt] (0,0.75)--(0.75,0)--(3,0)--(3,1.75)--(1.75,3)--(0,3)--(0,0.75);

\draw (4,2.5) node {\small{$\frac{1}{p}+\frac{1}{q}=1+\frac{1}{r}+\frac{d-2}{d}$}};

\end{tikzpicture}}
\end{center}
    
    \item \textit{Case 3:}$\quad \frac{1}{r}< \frac{d-2}{d(d-1)}$.
    
     In this case, $(\frac{1}{q}, \frac{1}{r})\in \text{int}(\Delta(d))$ with $ \frac{1}{r}>\frac{1}{q}-\frac{d-2}{d} $  if and only if $\frac{1}{r}<\frac{1}{q}<\frac{d}{r}$.
    
           \begin{center}
         \scalebox{0.9}{
\begin{tikzpicture}

\fill[blue!10!white] (0,0.25)--(0.25,0)--(3,0)--(3,0.75)--(0.75,3)--(0,3)--(0,0.25);
\fill[red!30!white] (0,0.25)--(0,0.75)--(3,0.75)--(3,0.25)--(0,0.25);
\draw[red,dashed,line width=1pt](0,0.25)--(0,0.75)--(3,0.75)--(3,0.25)--(0,0.25);

\draw (-0.3,3.8) node {$\frac{1}{q}$};
\draw (3.8,0) node {$\frac{1}{p}$};
\draw (3,-0.4) node {$1$};
\draw (-0.4,3) node {$1$};
\draw (0.25,-0.4) node {$\frac{1}{r}$};
\draw[red] (-1.6,0.75) node {$\frac{d}{r}$};
\draw (0.75,-0.5) node {$\frac{d}{r}$};
\draw[red] (-1.6,0.25) node {$\frac{1}{r}$};
\draw(-0.6,0.5) node {\small{$\frac{d-2}{d(d-1)}$}};

\draw[->,line width=1pt] (-0.2,0)--(3.5,0);
\draw[->,line width=1pt] (0,-0.2)--(0,3.5);
\draw[-,line width=1pt, red, dashed] (-1.2,0.25)--(3.4,0.25);
\draw[-,line width=1pt, red, dashed] (-1.2,0.75)--(3.4,0.75);

\draw[-,line width=1pt] (3,-0.1)--(3,0.1);
\draw[-,line width=1pt] (-0.1,3)--(0.1,3);
\draw[-,line width=1pt] (-0.1,0.75)--(0.1,0.75);
\draw[-,line width=1pt] (0.25,-0.1)--(0.25,0.1);
\draw[-,line width=1pt] (0.75,-0.1)--(0.75,0.1);

\draw[-,line width=1pt] (-0.1,0.5)--(0.1,0.5);

\filldraw[black] (0,0) circle (1.5pt);

\draw[dashed,blue,line width=1pt] (0,0.25)--(0.25,0)--(3,0)--(3,0.75)--(0.75,3)--(0,3)--(0,0.25);

\draw (2.8,2.5) node {\small{$\frac{1}{p}+\frac{1}{q}=1+\frac{d}{r}$}};

\end{tikzpicture}}
\end{center}
\end{enumerate}

So far multilinear interpolation for the parameters $p,q$ led us to 
\begin{equation}\label{rbig}
    \|\tilde{\mathcal{A}}_{\kappa}(f,g-\tau_h g)\|_r\lesssim |h|^{\eta} \|f\|_p\|g\|_q
\end{equation}
when $1/r<1/p+1/q<m(d,r)$ and $1/r<\frac{d-1}{d}$.

Notice that until this point the parameter $r$ hasn't played a role in the multilinear interpolation, and we still didn't prove the continuity estimate (\ref{hestimate2}) for the case $\frac{d-1}{d}\leq \frac{1}{r}<\frac{2d-1}{d}$. To address this, take $(p,q,r)$ with $\frac{d-1}{d}\leq \frac{1}{r}<\frac{2d-1}{d}$ and $\frac{1}{r}<\frac{1}{p}+\frac{1}{q}<\frac{2d-1}{d}=m(d,r)$. We want to show that 
$$\|\tilde{\mathcal{A}}_{\kappa}(f,g-\tau_h g)\|_{r}\lesssim |h|^{\eta}\|f\|_{p}\|g\|_{q}.$$

Denote $s:=\frac{1}{p}+\frac{1}{q}\in (\frac{1}{r},\frac{2d-1}{d})$. 

\begin{figure}[h]
 \begin{center}
         \scalebox{0.9}{
\begin{tikzpicture}

\fill[blue!10!white] (0,2.5)--(2.5,0)--(3,0)--(3,2)--(2,3)--(0,3)--(0,2.5);

\draw (-0.3,3.7) node {$y$};
\draw (3.7,-0.1) node {$x$};
\draw (3,-0.4) node {$1$};
\draw (-0.4,3) node {$1$};
\draw (2.5,-0.4) node {$\frac{1}{r}$};
\draw (-0.4,2.5) node {$\frac{1}{r}$};
\draw (-0.5,2) node{$\frac{d-1}{d}$};
\draw (1.3,-0.4) node{$\frac{1}{d}$};
\draw (-0.4,1.3) node{$\frac{1}{d}$};
\draw (2,-0.4) node{$\frac{d-1}{d}$};

\draw[->,line width=1pt] (-0.2,0)--(3.5,0);
\draw[->,line width=1pt] (0,-0.2)--(0,3.5);

\draw[-,line width=1pt] (3,-0.1)--(3,0.1);
\draw[-,line width=1pt] (-0.1,3)--(0.1,3);
\draw[-,line width=1pt] (-0.1,2)--(0.1,2);
\draw[-,line width=1pt] (2,-0.1)--(2,0.1);
\draw[-,line width=1pt] (-0.1,2.5)--(0.1,2.5);
\draw[-,line width=1pt] (2.5,-0.1)--(2.5,0.1);

\draw[-,line width=1pt] (1.3,-0.1)--(1.3,0.1);
\draw[-,line width=1pt] (-0.1,1.3)--(0.1,1.3);

\draw[-,line width=1pt, red] (3,3/4)--(3/4,3);
\draw[red] (4,0.8) node {\small{$x+y=s$}};
\draw[orange] (4.2,0.3) node {\small{$x+y=\frac{1}{r_2}$}};
\draw[green!70!black] (-0.2,0.8) node {\small{$x+y=\frac{1}{r_1}$}};

\draw[-,line width=1pt, orange] (3,0.2)--(0.2,3);
\draw[-,line width=1pt, green!70!black] (1.7,0)--(0,1.7);

\filldraw[black] (0,0) circle (1.5pt);
\filldraw[red] (9/4,3/2) circle (1.5pt);
\draw[red] (2.7,1.7) node {$(\frac{1}{p},\frac{1}{q})$};
\draw[dashed,blue,line width=1pt] (0,2.5)--(2.5,0)--(3,0)--(3,2)--(2,3)--(0,3)--(0,2.5);

\end{tikzpicture}}
\end{center}
\caption{Figure illustrating how interpolation in the parameter $r$ gives us continuity estimates when $\frac{d-1}{d}\leq \frac{1}{r}<\frac{2d-1}{d}$.}
\end{figure}
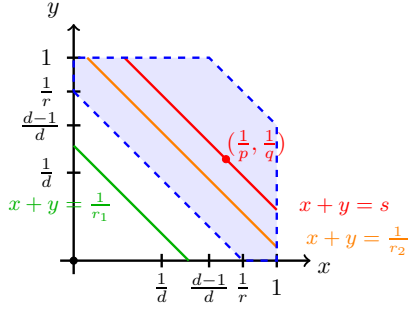
Pick $0<r_1,r_2< \infty$ satisfying $\frac{d}{2d-1}<r_2<r\leq \frac{d}{d-1}<r_1< d$. We may also choose $r_2$ such that $\frac{1}{r_2}<s $. Since $\frac{1}{d}<\frac{1}{r_1}<\frac{d-1}{d}$, and $\frac{1}{r_1}<s=\frac{1}{p}+\frac{1}{q}<\frac{2d-1}{d}=m(d,r_1)$, one has by (\ref{rbig}) that
$$\|\tilde{\mathcal{A}}_{\kappa}(f,g-\tau_h g)\|_{r_1}\lesssim |h|^{\eta_1}\|f\|_{p}\|g\|_{q}.$$

Also, since we chose $r_2$ so that $\frac{1}{r_2}<s$, $(1/p,1/q)$ satisfies $\frac{1}{r_2}<\frac{1}{p}+\frac{1}{q}<\frac{2d-1}{d}=m(d,r_2)$, and (\ref{nohgain}) gives the boundedness estimate
$$\|\tilde{\mathcal{A}}_{\kappa}(f,g-\tau_h g)\|_{r_2}\lesssim \|f\|_{p}\|g\|_{q}.$$

Since $r_2<r<r_1$, we get by interpolating those two estimates that 
$$\|\tilde{\mathcal{A}}_{\kappa}(f,g-\tau_h g)\|_{r}\lesssim |h|^{\eta}\|f\|_{p}\|g\|_{q} $$
where $\frac{1}{r}=(1-\theta)\frac{1}{r_1}+\theta \frac{1}{r_2}$, and $\eta=(1-\theta)\eta_1 $.

\end{proof}

\begin{proof}[Proof of Lemma \ref{hlemma}]
By (\ref{split2}) and H\"older's inequality,
\begin{equation*}
    \begin{split}
        \|\tilde{\mathcal{A}}_{\kappa}(f,g-\tau_h g)\|_r\lesssim & \||f|*\chi_B(x)\sum_{l=0}^{\infty}2^{-l(\frac{d-2}{2})} \tilde{\mathcal{S}}[(g-\tau_h g)(2^{-l/2}\cdot)](2^{l/2}x)\|_r\\
        \leq &\||f|*\chi_B\|_{\infty} \sum_{l=0}^{\infty}2^{-l(\frac{d-2}{2})}\|\tilde{\mathcal{S}}[(g-\tau_h g)(2^{-l/2}\cdot)](2^{l/2}x)\|_r\\
        \leq &\|f\|_p \sum_{l=0}^{\infty}2^{-l(\frac{d-2}{2})}2^{-l/2\cdot d/r}\|\tilde{\mathcal{S}}[(g-\tau_h g)(2^{-l/2}\cdot)]\|_r.
    \end{split}
\end{equation*}

Denote $g_l(x)=g(2^{-\frac{l}{2}}x)$, then 
$$(\tau_h g)(2^{\frac{-l}{2}}x)=g(2^{\frac{-l}{2}}x-h)=g(2^{\frac{-l}{2}}(x-2^{\frac{l}{2}}h))=\tau_{2^{\frac{l}{2}}h}g_l(x)$$
and we get
\begin{equation*}
\begin{split}
    \|\tilde{\mathcal{A}}_{\kappa}(f,g-\tau_h g)\|_r\lesssim  \|f\|_p \sum_{l=0}^{\infty}2^{-\frac{l}{2}(d-2+\frac{d}{r})}\|\tilde{\mathcal{S}}[g_l-\tau_{2^{\frac{l}{2}}h} g_l]\|_r.
\end{split}
\end{equation*}

Now we consider two cases for $l\geq 0$. We will prove a continuity estimate for each and then choose $\eta$ to be the minimum of the values in the cases.

\begin{itemize}
    \item \textit{First case:} $|2^{\frac{l}{2}}h|< 1$, i.e., $l< 2 \log_2(\frac{1}{|h|})$. Denote $L:=\lfloor2 \log_2(\frac{1}{|h|})\rfloor$.
    
    For this case, we use the known continuity estimates for $\tilde{\mathcal{S}}$, proved in \cite{Lacey}. Since $(1/q,1/r)\in \textnormal{int}(\Delta(d))$, one has
    \begin{equation*}
\begin{split}
    \|\tilde{\mathcal{S}}[g_l-\tau_{2^{\frac{l}{2}}h} g_l]\|_r\lesssim 2^{\frac{l\eta}{2}}|h|^{\eta} \| g_l\|_q=2^{\frac{l\eta}{2}}|h|^{\eta} 2^{\frac{l}{2}\cdot \frac{d}{q}}\| g\|_q.
\end{split}
\end{equation*}
Thus,
\begin{equation*}
    \begin{split}
        \|f\|_p\sum_{l=0}^{L} 2^{-\frac{l}{2}(d-2+d/r)}&\|\tilde{\mathcal{S}}[g_l-\tau_{2^{\frac{l}{2}}h} g_l]\|_r\\
        &\lesssim  |h|^{\eta}\|f\|_p\sum_{l=0}^{L} 2^{-\frac{l}{2}(d-2+d/r-d/q-\eta)}\|g\|_q\\
       & \lesssim |h|^{\eta}\|f\|_p\|g\|_q,
    \end{split}
\end{equation*}
    where we used that by hypothesis $d-2+d/r-d/q>0$.

    \item \textit{Second case:} $|2^{\frac{l}{2}}h|\geq 1$, i.e., $l\geq  2 \log_2(\frac{1}{|h|})$.
    
    Since $(1/q,1/r)\in \text{int}(\Delta(d))$, by the known bounds for the localized operator $\tilde{\mathcal{S}}$ in Proposition \ref{linearlpimproving} one has
    
    \begin{equation*}
    \begin{split}
        &\|f\|_p\sum_{l\geq L} 2^{-\frac{l}{2}(d-2+d/r)}\|\tilde{\mathcal{S}}[g_l-\tau_{2^{\frac{l}{2}}h} g_l]\|_r\\
        \lesssim&\|f\|_p\sum_{l\geq L} 2^{-\frac{l}{2}(d-2+d/r)}\left(\|\tilde{\mathcal{S}}g_l\|_r+\|\tilde{\mathcal{S}}[\tau_{2^{\frac{l}{2}}h} g_l]\|_r\right)\\
        \lesssim & \|f\|_p\sum_{l\geq L} 2^{-\frac{l}{2}(d-2+d/r)}\|g_l\|_q\\
         \lesssim & \|f\|_p\sum_{l\geq L} 2^{-\frac{l}{2}(d-2+d/r-d/q)}\|g\|_q\\
         \lesssim &2^{-\frac{LM}{2} }\|f\|_p\|g\|_q,
         \end{split}
\end{equation*}where $M=d-2+d/r-d/q>0$. Hence, this term is bounded by $|h|^{M}\|f\|_p \|g\|_q.$
\end{itemize}

\end{proof}
\begin{rem}
One can simplify the interpolation argument in the proof of Proposition \ref{continuity} by directly using interpolation for multi-sublinear forms of generalized restricted type. Since the region of known $L^p$ improving estimate $1/r<1/p+1/q<m(d,r)$ is a convex domain in $\mathbb{R}^3$ (being the intersection of half-spaces) and $p, q>1$, the generalized restricted type bounds obtained by bilinear interpolation automatically upgrade to the desired $L^p$ space bounds. We refer to \cite[Chapter 3]{Thiele} and \cite[Proposition 2.2]{BOS} for details of multilinear interpolation. 
\end{rem}

\begin{rem} If we revisit the proof of Lacey's continuity estimate in the case $d\geq 3$, one can see that one can actually get continuity estimates in the open segment in the boundary of $\Delta(d)$ that connects $(0,0)$ to $(\frac{d-1}{d}, \frac{d-1}{d})$. That is because the continuity estimate is proved at $(1/2,1/2)$ which lives in the interior of that segment (the same is not true for $d=2$). This can allow us to improve the continuity estimates to include $\frac{1}{p}+\frac{1}{q}=\frac{1}{r}$, but for simplicity we will be mostly interested in proving the sparse domination bounds in an open region. Indeed, the boundary points of the range of exponents usually don't seem to change most of the implications of sparse bounds.
\end{rem}

We now consider the $d=2$ case of the continuity estimate. Note that the argument in the $d\geq 3$ case above doesn't apply anymore, as one would end up with a divergent sum. This technical difficulty stems from the fact that the spherical averaging operator becomes more singular as the dimension gets lower. We first collect two preliminary lemmas relating to linear spherical maximal functions.
\begin{lem}\label{lem: annulus}
Let $d\geq 2$, and $C\geq 1$. For exponents satisfying $(1/q,1/s)\in \text{int}(\Delta(d))$, we have the bound
\begin{equation}
    \left\Vert \sup_{t\in [1,2]} \int_{1-C|h|\le |y|\le 1+C|h|} |g(x-ty)|\,dy\right\Vert_{L^s} \lesssim |h| \Vert g\Vert_{L^q},\,\text{for all }|h|\ll1.
\end{equation}
\end{lem}
\begin{proof}
Let $\mathcal{S}_{L}$ denote the local spherical maximal operator across scales $[1/2,3]$, for instance.
That is, 
$$\mathcal{S}_{L}(g)(x)=\sup_{t\in [1/2,3]} \int_{S^{d-1}}|g(x-ty)|dy.$$ Obviously, $\mathcal{S}_L$ satisfies the same $L^q\to L^s$ bound as $\tilde{\mathcal{S}}$ (as defined in \eqref{loc sphere}).

Then, using polar coordinates, we have for any $|h|\leq \frac{1}{2C}$,
\begin{align*}
    &\left\Vert \sup_{t\in [1,2]} \int_{1-C|h|\le |y|\le 1+C|h|} |g(x-ty)|\,dy\right\Vert_{L^s} \\
    =& \left\Vert \sup_{t\in [1,2]} \int_{1-C|h|}^{1+C|h|}\int_{S^{d-1}} |g(x-tr\theta)|r^{d-1}\,d\sigma(\theta)\,dr\right\Vert_{L^s} \\
    \lesssim& \|\sup_{t\in [1,2]} \int_{1-C|h|}^{1+C|h|}\mathcal{S}_{L}g(x)\,dr\|_{L^s}\\
    \lesssim & \left\Vert |h| \mathcal{S}_Lg\right\Vert_{L^s} \lesssim |h| \Vert g\Vert_{L^q}.
\end{align*}
Here, we used that if $t\in [1,2]$, $r\in [1-C|h|,1+C|h|]$, and $|h|\leq \frac{1}{2C}$, then $tr\in [1/2,3]$.
\end{proof}

We also will find it useful to rescale a result of Lacey \cite[Proposition 4.2]{Lacey}.
\begin{lem}\label{lemma2d}
Fix $(\frac{1}{p},\frac{1}{r})$ in the interior of $F_2'=\Delta(2)$. Then there exists $\eta>0$ such that for all $0<\gamma<1/2$ and all $0<\epsilon \leq 1$, we have
\begin{equation}
    \left\Vert \sup_{s,t\in[1,2],|s-t|<\gamma} |A_{\epsilon s}f-A_{\epsilon t}f|\right\Vert_{L^r} \lesssim \gamma^{\eta}\epsilon^{2/r-2/p}\Vert f\Vert_{L^p}.
\end{equation}
\end{lem}
\begin{proof}
From \cite[Proposition 4.2]{Lacey}, we know that this holds when $\epsilon=1$. Adopt the notation $\delta_{\epsilon}g(u)=g(\epsilon u)$ for convenience. Then we have the identity
\begin{equation}
    |A_{\epsilon s}f-A_{\epsilon t}f|(x)=|A_s\delta_{\epsilon}f-A_t\delta_{\epsilon}f|(x/\epsilon).
\end{equation}
Taking the supremum, taking an $L^r$ norm and changing variables, we can apply the scale $1$ result of Lacey to deduce
\begin{equation}
    \left\Vert \sup_{s,t\in[1,2],|s-t|<\gamma} |A_{\epsilon t}f-A_{\epsilon s}f|\right\Vert_r \lesssim \gamma^{\eta}\epsilon^{2/r}\Vert \delta_{\epsilon}f\Vert_p.
\end{equation}
Then changing variables on the right side of the estimate immediately gives the claim.
\end{proof}

Now, we are ready to prove the key two-dimensional estimate. The proposition below shows that one still has a quite satisfactory continuity estimate when $d=2$.

\begin{prop}\label{continuity2} ($\eta$ regularity condition for the bilinear spherical maximal operator for $d= 2$) Let $d= 2$, $1\leq p,q\leq \infty$, $0<r< \infty$. One has that there exists $\eta=\eta(p,q,r)>0$ such that
\begin{equation}\label{hestimate3}
    \left\|\sup_{1\leq t\leq 2}|\mathcal{A}_t(f,g-\tau_h g)|\right\|_{L^r}\leq C |h|^{\eta}\|f\|_{L^p}\|g\|_{L^q}
\end{equation}
for all $|h|<1$,  provided that $r>\dfrac{2}{3}$ and
\begin{equation*}
    \begin{split}
        \dfrac{1}{r}< \dfrac{1}{p}+\dfrac{1}{q}< \min\left\{1+\dfrac{1}{r}, \dfrac{3}{2}\right\}=\begin{cases}
    &3/2,\text{ if } r\leq 2,\\
    &1+1/r,\text{ if }r>2
     \end{cases}
    \end{split}.
\end{equation*}

\end{prop}

\begin{proof}
In the previous notation, our goal is to prove estimates of the form
\begin{equation}
   \|\tilde{\mathcal{A}}_{\kappa}(f,g-\tau_h g)\|_r\lesssim |h|^{\eta}\|f\|_p\|g\|_q 
\end{equation}
in the case $d=2$.

We start by making a change of variables $y'= y+\frac{h}{\kappa(x)}$ in the integral involving the translation of $g$, from which we get
\begin{equation*}
    \begin{split}
        &\tilde{\mathcal{A}}_{\kappa}(f,g-\tau_h g)(x)\\
        =&\int_{B^{2}(0,1)} (g-\tau_h g)(x-\kappa(x)y)\int_{S^{1}} f(x-\kappa(x)\sqrt{1-|y|^2}z)d\sigma (z) dy\\
        =&\int g(x-\kappa(x)y)\Bigg\{\chi_{B^{2}(0,1)}(y)\int_{S^{1}}  f(x-\kappa(x)\sqrt{1-|y|^2}z)d\sigma(z)\\
        &-\chi_{B^2(0,1)}\left(y-\frac{h}{\kappa(x)}\right)\int_{S^{1}} f(x-\kappa(x)\sqrt{1-|y-h/{\kappa(x)}|^2}z) d\sigma (z)\Bigg\} dy\\
        =&\int g(x-\kappa(x)y)\{ \chi_{B^{2}(0,1)}(y) A_{\kappa(x)\sqrt{1-|y|^2}}(f)(x)\\
        &-\chi_{B^2(0,1)}(y-h/{\kappa(x)}) A_{\kappa(x)\sqrt{1-|y-h/\kappa(x)|^2}}(f)(x) \} dy,
    \end{split}
\end{equation*}
where $A_t(f)(x)=\int_{S^{1}} f(x-ty)d\sigma (y)$.

If $|y|>1+|h|$, then $\chi_{B^{2}(0,1)}(y)=0$ and $\chi_{B^{2}(0,1)}(y-\frac{h}{\kappa(x)})=0$ (because $\kappa(x)\in [1,2]$).

If $|y|<1-|h|$, then $\chi_{B^{2}(0,1)}(y)=1$ and $\chi_{B^{2}(0,1)}(y-\frac{h}{\kappa(x)})=1$.

Denote $\tilde{h}(x)=\frac{h}{\kappa(x)}$. Then, for a constant $C\geq 1$ to be chosen, the quantity above can be rewritten as $I_h(f,g)(x)+II_h(f,g)(x)$, where
\begin{equation*}
    \begin{split}
        &I_h(f,g)(x)\\
        :=&\int_{1-C|h|\leq|y|\leq 1+|h|} g(x-\kappa(x)y)\{\chi_{B^{2}(0,1)}(y) A_{\kappa(x)\sqrt{1-|y|^2}}(f)(x)\\
        &\quad-\chi_{B^2(0,1)}(y-\tilde{h}) A_{\kappa(x)\sqrt{1-|y-\tilde{h}|^2}}(f)(x)\} \,dy
        \end{split}
        \end{equation*}and
\begin{equation*}
    \begin{split}
&II_h(f,g)(x)\\
        :=&\int_{B^2(0,1-C|h|)} g(x-\kappa(x)y)\{ A_{\kappa(x)\sqrt{1-|y|^2}}(f)(x)- A_{\kappa(x)\sqrt{1-|y-\tilde{h}|^2}}(f)(x)\} dy.
    \end{split}
\end{equation*}

We observe that for $h$ fixed, $I_h(f,g)$ and $II_h(f,g)$ are both bilinear operators that inherit the boundedness properties of $\tilde{\mathcal{M}}$. This follows from the following inequalities:
\begin{equation*}
    \begin{split}
        &|I_h(f,g)(x)|\\
        \leq & \int_{1-C|h|\leq |y|\leq 1} |g(x-\kappa(x)y)|\int_{S^{1}} |f(x-\kappa(x)\sqrt{1-|y|^2}z)|\,d\sigma(z) dy\\
        &+ \int |g(x-\kappa(x)y)| \chi_{B^2(0,1)} (y-\tilde{h})\cdot \\
        &\qquad\qquad\int_{S^{1}} |f(x-\kappa(x)\sqrt{1-|y-\tilde{h}|^2}z)| \,d\sigma(z)dy\\
        \leq & \tilde{\mathcal{A}}_{\kappa}(|f|,|g|)(x)\\
        &\quad +\int_{B^2(0,1)} |(\tau_hg)(x-\kappa(x)y')|\int_{S^{1}} |f(x-\kappa(x)\sqrt{1-|y'|^2}z)|\,d\sigma(z)dy'\\
        \leq & \tilde{\mathcal{A}}_{\kappa}(|f|,|g|)(x)+\tilde{\mathcal{A}}_{\kappa}(|f|,|\tau_h g|)(x),
    \end{split}
\end{equation*}and similarly
\begin{equation*}
    \begin{split}
        &|II_h(f,g)(x)|\\
        \leq & \int_{B^{2}(0,1-C|h|)} |g(x-\kappa(x)y)|\int_{S^{1}} |f(x-\kappa(x)\sqrt{1-|y|^2}z)|\,d\sigma(z) dy\\
        &+ \int_{B^{2}(0,1-C|h|)} |g(x-\kappa(x)y)| \int_{S^{1}} |f(x-\kappa(x)\sqrt{1-|y-\tilde{h}|^2}z)| \,d\sigma(z)dy\\
        \leq & \tilde{\mathcal{A}}_{\kappa}(|f|,|g|)(x)\\
        &\quad +\int_{B^2(0,1)} |(\tau_hg)(x-\kappa(x)y')|\int_{S^{1}} |f(x-\kappa(x)\sqrt{1-|y'|^2}z)|\,d\sigma(z)dy'\\
        \leq & \tilde{\mathcal{A}}_{\kappa}(|f|,|g|)(x)+\tilde{\mathcal{A}}_{\kappa}(|f|,|\tau_h g|)(x).
    \end{split}
\end{equation*}

We will show that $\|I_h(f,g)\|_r\lesssim |h|^{\eta}\|f\|_p\|g\|_q$, for some triples $(1/p,1/q,1/r)$ and get the rest via interpolation with the boundedness estimates of $\tilde{\mathcal{A}}_{\kappa}$. Similarly with $II_h(f,g)$.

Since $d=2$, $\Delta(2)$ is a triangle with vertices
\begin{equation} \label{vertices2}
\mathcal{B}_1^2=(0,0),\,\mathcal{B}_2^2=\left(\frac{1}{2},\frac{1}{2}\right)=\mathcal{B}_3^2,\text{ and }\mathcal{B}_4^2=\left(\frac{2}{5},\frac{1}{5}\right).
\end{equation}

Observe that 
\begin{equation*}
    \begin{split}
        |I_h(f,g)(x)|\leq &2\int_{1-C|h|\leq |y|\leq 1+|y|} |g(x-\kappa(x)y)| \sup_{t>0} A_t(f)(x) \,dy\\
        \leq& 2\mathcal{M}_{linear}(f)(x)\sup_{t\in[1,2]}\int_{1-C|h|
        \leq  |y|\leq 1+C|h|} |g(x-t y)|\,dy.
    \end{split}
\end{equation*}

Hence, for $s$ given by the H\"older relation $1/r=1/p+1/s$, we get from Lemma \ref{lem: annulus} that 
\begin{equation*}
    \begin{split}
        \|I_h(f,g)\|_r\leq & \|\mathcal{M}_{linear}(f)\|_{p}\left\Vert\sup_{t\in[1,2]}\int_{1-C|h|
        \leq  |y|\leq 1+C|y|} |g(x-t y)|\,dy\right\Vert_{L_x^s}\\
        \lesssim & |h|\|f\|_p \|g\|_q,
        \end{split}
        \end{equation*}if $p>2$, and $\left(\frac{1}{q},\frac{1}{s}\right)\in \text{int}(\Delta(2))$. The condition $p>2$ comes from the boundedness properties of $\mathcal{M}_{\text{linear}}$ for $d=2$ \cite{Bourgain}.

In terms of $(p,q,r)$ we are asking $(\frac{1}{q},\frac{1}{r}-\frac{1}{p})\in \text{int} (\Delta(2))$, and $\frac{1}{p}<\frac{1}{2}$.

 We can prove Proposition \ref{continuity2} assuming $r$ is sufficiently large, say $r>5$. Then using the same trick at the end of the proof of Proposition 
     \ref{continuity}, namely, multilinear interpolation in the three parameters $p,q,r$ of the known continuity estimates for $r$ large, and the boundedness estimates (no power of $|h|$ gain) we can get the result for small $r$ as well.
     
Recalling the definition of $\Delta(2)$ in \eqref{vertices2}, we know that a point $(x,y)\in \textnormal{int}(\Delta(2))$ if and only if $y<x$, $y>x/2$, and $y>3x-1$.
     For $y<1/5$ the condition $y>3x-1$ does not add restriction. 

When $r>5$, one has $\frac{1}{r}-\frac{1}{p}\leq \frac{1}{5}$ so 
$$\left(\frac{1}{q},\frac{1}{r}-\frac{1}{p}\right)\in \text{int} (\Delta(2))\iff  \frac{1}{r}-\frac{1}{p}<\frac{1}{q}<2\left(\frac{1}{r}-\frac{1}{p}\right).$$

That is,
$$\frac{1}{r}<\frac{1}{p}+\frac{1}{q}\quad \text{ and }\quad \frac{1}{2q}+\frac{1}{p}<\frac{1}{r}.$$

This means that the pair $(1/p,1/q)$ is in the region between the red and the orange lines in the figure below. Note that the constraint $1/p<1/2$ is automatically satisfied for a point in this region since we are assuming $r>5$.

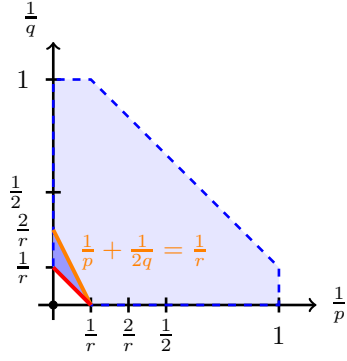
\begin{figure}[h]
\begin{center}
         \scalebox{1}{
\begin{tikzpicture}

\fill[blue!10!white] (0,0.5)--(0.5,0)--(3,0)--(3,0.5)--(0.5,3)--(0,3)--(0,0.5);

\draw (-0.3,3.8) node {$\frac{1}{q}$};
\draw (3.8,0) node {$\frac{1}{p}$};
\draw (3,-0.4) node {$1$};
\draw (-0.4,3) node {$1$};
\draw (0.5,-0.4) node {$\frac{1}{r}$};
\draw (-0.4,0.5) node {$\frac{1}{r}$};
\draw (-0.5,1.5) node{$\frac{1}{2}$};
\draw (1.5,-0.4) node{$\frac{1}{2}$};
\draw (1,-0.4) node {$\frac{2}{r}$};
\draw (-0.4,1) node {$\frac{2}{r}$};

\draw[->,line width=1pt] (-0.2,0)--(3.5,0);
\draw[->,line width=1pt] (0,-0.2)--(0,3.5);

\draw[-,line width=1pt] (3,-0.1)--(3,0.1);
\draw[-,line width=1pt] (-0.1,3)--(0.1,3);
\draw[-,line width=1pt] (-0.1,0.5)--(0.1,0.5);
\draw[-,line width=1pt] (0.5,-0.1)--(0.5,0.1);
\draw[-,line width=1pt] (-0.1,1.5)--(0.1,1.5);
\draw[-,line width=1pt] (1.5,-0.1)--(1.5,0.1);
\draw[-,line width=1pt] (1,-0.1)--(1,0.1);

\filldraw[blue!40!white] (0,0.5)--(0.5,0)--(0,1)--(0,0.5);
\draw[dashed,blue,line width=1pt] (0,0.5)--(0.5,0)--(3,0)--(3,0.5)--(0.5,3)--(0,3)--(0,0.5);

\draw[-,line width=1.5pt, orange] (0.5,0)--(0,1);
\draw[-,line width=1.5pt, red] (0,0.5)--(0.5,0);

\draw[orange] (1.2,0.7) node {$\frac{1}{p}+\frac{1}{2q}=\frac{1}{r}$};

\filldraw[black] (0,0) circle (1.5pt);

\end{tikzpicture}}
\end{center}
\caption{Figure illustrating the initial region for which we get continuity estimates for $I_{h}(f,g)$.}
\end{figure}

Interpolating with the boundedness estimates we get 
$$\|I_h(f,g)\|_r\lesssim |h|^{\eta}\|f\|_p\|g\|_q$$
for all $1/r<1/p+1/q<1+1/r=m(2,r)$.

To estimate the $L^r$ norm of $II_h(f,g)$ we start by using the Minkowski inequality for integrals:
\begin{equation}
    \begin{split}
        \|II_h(f,g)\|_{r}\leq& \int_{B^2(0,1-C|h|)}\|g(x-\kappa(x)y)(A_{\kappa(x)\sqrt{1-|y|^2}} f-A_{\kappa(x)\sqrt{1-|y-\tilde{h}|^2}}f)\|_{L^r_x} \,dy\\
        \leq &\|g\|_{\infty}\int_{B^2(0,1-C|h|)} \|A_{\kappa(x)\sqrt{1-|y|^2}} f(x)-A_{\kappa(x)\sqrt{1-|y-\tilde{h}|^2}}f(x)\|_{L^{r}_x}\, dy.
    \end{split}
\end{equation}
    
     We will now split the $y$ being considered into annuli where $(1-|y|)$ lives at scale $2^{-l}$ as in \cite{JL}:
     $$ \mathbb{A}_l=\{y\in B^2(0,1)\colon 1-2^{-l-1}\leq |y|\leq 1-2^{-l-2}\}$$
     for $l\geq 1$ and $\mathbb{A}_0=\{y\in B(0,1)\colon |y|\leq 3/4\}$.
     
      In two dimensions, each annulus has area comparable to its scale, i.e.
    $$|\mathbb{A}_l|\lesssim \int_{1-2^{-l-1}}^{1-2^{-l-2}} {\tau} \,d\tau=\frac{1}{2} \{(1-2^{-l-2})^2-(1-2^{-l-1})^2\}\lesssim 2^{-l}.$$
    
    We observe that if $l$ is too large then $\mathbb{A}_l$ does not intersect $B^2(0,1-C|h|)$. Indeed if $y\in B^2(0,1-C|h|)\cap \mathbb{A}_l$ then 
    \begin{equation}
        1-2^{-l-1}\leq |y|\leq 1-C|h|\Rightarrow 2^{-l-1}\geq C|h|\Rightarrow l+1\leq  \log_2\left(\frac{1}{C|h|}\right).
    \end{equation}
    Hence if we define $L_h=\lfloor\log_2\left(\frac{1}{C|h|}\right)\rfloor$ one has 
    \begin{equation}
        B^2(0,1-C|h|)=\bigcup_{l=0}^{L_h}\mathbb{A}_l \cap B(0,1-C|h|) 
    \end{equation}

    Denoting $B=B^2(0,1-C|h|)$, using the splitting above, we can break the integral into pieces
    \begin{equation*}
        \begin{split}
            \|II_h(f,g)\|_r\lesssim \|g\|_{\infty} \sum_{l=0}^{L_h}\int_{\mathbb{A}_l\cap B} \|A_{\kappa(x)\sqrt{1-|y|^2}}f(x)-A_{\kappa(x)\sqrt{1-|y-\tilde{h}|^2}}f(x)\|_{L^r_x} \,dy.
        \end{split}
    \end{equation*}
    
    Observe that since $l\leq L_h$, if $C > 8$ then one has $|h|\leq \frac{2^{-l}}{C} <2^{-l-3}$, so for all $y\in \mathbb{A}_l$ one has  that $\sqrt{1-|y|^2}$ and $\sqrt{1-|y-\tilde{h}|^2}$ are both comparable to $2^{-l/2}$ .

    Indeed, 
    $$\sqrt{1-|y|^2}=\sqrt{1+|y|}\sqrt{1-|y|}\sim 2^{-l/2},$$
    
    and 
   $$ 1-2^{-l}\leq |y-\tilde{h}|\leq 1-2^{-l-3}\implies \sqrt{1-|y-\tilde{h}|^2}\sim \sqrt{1-|y-\tilde{h}|} \sim 2^{-l/2}.$$

    By Taylor expansion, there is some universal constant $C_2$ such that
    $$\left|\dfrac{1}{2^{-l/2}}\sqrt{1-|y|^2}-\dfrac{1}{2^{-l/2}}\sqrt{1-|y-\tilde{h}|^2}\right|\leq C_22^{l/2}2^{l/2}|\tilde{h}|\leq C_22^{l}|h|=:\gamma<\frac{1}{2},$$
with the last inequality following as long as we choose $C> 2C_2$. Thus, we can apply Lemma \ref{lemma2d} to get that if $(1/p,1/r)$ is in the interior of $\Delta(2)$ then $\|II_h(f,g)\|_{r}$ is dominated by
    \begin{equation*}
        \begin{split}
        \lesssim &\|g\|_{\infty} \sum_{l=0}^{L_h} (|h|2^l)^{\eta} (2^{-l/2})^{2(\frac{1}{r}-\frac{1}{p})}|\mathbb{A}_l| \|f\|_p \\
             \lesssim &\|g\|_{\infty}\|f\|_p \sum_{l=0}^{L_h} |h|^{\eta} 2^{-l(\frac{1}{r}-\frac{1}{p}-\eta)}|\mathbb{A}_l| \\
             \lesssim &\|g\|_{\infty}\|f\|_p |h|^{\eta}\sum_{l=0}^{L_h}  2^{-l(1+\frac{1}{r}-\frac{1}{p}-\eta)}\\
             \lesssim & |h|^{\eta}\|f\|_p\|g\|_{\infty}
        \end{split}
    \end{equation*}
    since we have that $1+\dfrac{1}{r}-\dfrac{1}{p}\geq \dfrac{1}{r}>0$.
    
    In conclusion, we have that 
    $$\|II_h(f,g)\|_r\lesssim |h|^{\eta}\|f\|_p\|g\|_{\infty}$$
     for exponents satisfying $(1/p,1/r)\in \text{int}(\Delta(2))$.

     Note that for $r>5$, one has $(1/p,1/r)\in \text{int}(\Delta(2))$ if and only if $\frac{1}{r}<\frac{1}{p}<\frac{2}{r}$.
     
     \begin{figure}[h]
\begin{center}
         \scalebox{1}{
\begin{tikzpicture}

\fill[blue!10!white] (0,0.5)--(0.5,0)--(3,0)--(3,0.5)--(0.5,3)--(0,3)--(0,0.5);

\draw (-0.3,3.8) node {$\frac{1}{q}$};
\draw (3.8,0) node {$\frac{1}{p}$};
\draw (3,-0.4) node {$1$};
\draw (-0.4,3) node {$1$};
\draw (0.5,-0.4) node {$\frac{1}{r}$};
\draw (-0.4,0.5) node {$1/r$};
\draw (-0.5,1.5) node{$\frac{1}{2}$};
\draw (1.5,-0.4) node{$\frac{1}{2}$};
\draw (1,-0.4) node {$\frac{2}{r}$};
\draw (-0.4,1) node {$2/r$};

\draw[->,line width=1pt] (-0.2,0)--(3.5,0);
\draw[->,line width=1pt] (0,-0.2)--(0,3.5);

\draw[-,line width=1pt] (3,-0.1)--(3,0.1);
\draw[-,line width=1pt] (-0.1,3)--(0.1,3);
\draw[-,line width=1pt] (-0.1,0.5)--(0.1,0.5);
\draw[-,line width=1pt] (0.5,-0.1)--(0.5,0.1);
\draw[-,line width=1pt] (-0.1,1.5)--(0.1,1.5);
\draw[-,line width=1pt] (1.5,-0.1)--(1.5,0.1);
\draw[-,line width=1pt] (1,-0.1)--(1,0.1);

\draw[dashed,blue,line width=1pt] (0,0.5)--(0.5,0)--(3,0)--(3,0.5)--(0.5,3)--(0,3)--(0,0.5);
\draw[-,line width=1.5pt, orange] (0.5,0)--(1,0);
\filldraw[black] (0,0) circle (1.5pt);
\end{tikzpicture}}
\end{center} 
\caption{Figure illustrating the initial region for which we get continuity estimates for $II_{h}(f,g)$, namely the open segment in orange.}
\end{figure}
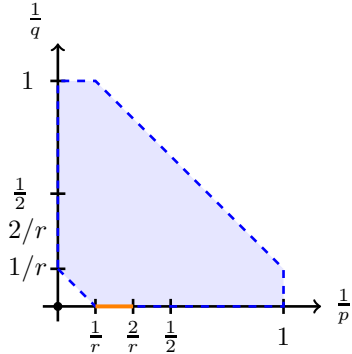
 Interpolating with the boundedness estimates for $II_h(f,g)$ we get the continuity estimates in the interior of the boundedness region of $\tilde{\mathcal{M}}$. Then
$$\|\tilde{\mathcal{A}_{\kappa}}(f,g)\|_r\leq \|I_h(f,g)\|_r+\|II_h(f,g)\|_r\lesssim |h|^{\eta}\|f\|_p\|g\|_q $$
for some $\eta(p,q,r)>0$ for any point in $1/r<1/p+1/q<1+1/r$.  
 \end{proof}

We now move onto the continuity estimate that involves translation in both input functions $f$ and $g$.
\begin{prop}\label{prop: double}
Let $d\geq 2$, and $(\frac{1}{p},\frac{1}{q},\frac{1}{r}) \in\mathcal{R}(d)$. Then there exists $\eta_1,\eta_2$ such that 
$$\|\tilde{\mathcal{M}}(f-\tau_{h_1}f,g-\tau_{h_2}g)\|_{L^r}\lesssim |h_1|^{\eta_1}|h_2|^{\eta_2}\|f\|_{L^p}\|g\|_{L^q}$$
for all $|h_1|<1,\,|h_2|<1$.
\end{prop}

\begin{proof}
Again, it suffices to prove this bound for the linearized operator $\tilde{\mathcal{A}_{\kappa}}$, where $\kappa:\R^d\rightarrow[1,2]$ is measurable.

For $(1/p,1/q,1/r)\in \mathcal{R}(d)$, take $(1/p_i,1/q_i,1/r)\in \mathcal{R}(d)$, $i\in \{1,2\}$, such that
$$(1/p,1/q)=\theta (1/p_1,1/q_1) +(1-\theta)(1/p_2,1/q_2).$$

From what we proved in Propositions \ref{continuity} and \ref{continuity2}, we know that 
\begin{equation*}
    \begin{split}
        \|\tilde{\mathcal{A}_{\kappa}}(f-\tau_{h_1}f,g-\tau_{h_2}g)\|_r&\leq \|\tilde{\mathcal{A}_{\kappa}}(f-\tau_{h_1}f,g)\|_r+\|\tilde{\mathcal{A}_{\kappa}}(f-\tau_{h_1}f,\tau_{h_2}g)\|_r\\
        &\lesssim |h_1|^{\eta_1}\|f\|_{p_1}\|g\|_{q_1}.
    \end{split}
\end{equation*}

 Similarly, using the linearity in the first entry, 
$$\|\tilde{\mathcal{A}_{\kappa}}(f-\tau_{h_1}f,g-\tau_{h_2}g)\|_r
\lesssim|h_2|^{\eta_2}\|f\|_{p_2}\|g\|_{q_2}.$$

Interpolating these estimates we get 
$$\|\tilde{\mathcal{A}_{\kappa}}(f-\tau_{h_1}f,g-\tau_{h_2}g)\|_r
\lesssim |h_1|^{\theta \eta_1}|h_2|^{(1-\theta)\eta_2}\|f\|_{p}\|g\|_{q}.$$

\end{proof}

\section{A digression to the $d=1$ case}\label{sec: d1}

In this section, we study the continuity estimate of the single scale bilinear averaging operator in dimension $d=1$. A completely different method is needed in this case, as the operator is much more singular and the slicing technique stops applying. More precisely, we shall get the continuity estimate as a consequence of a new trilinear smoothing theorem recently developed in \cite{MCZZ}. This theorem extends the earlier trilinear smoothing inequality from \cite{MC} to include a class of degenerate surfaces. In particular, the surface parametrized by $\phi_1 = x, \phi_2 = x-\cos t, \phi_3 = x-\sin t$, which is the surface in our problem, falls into this class.

Let $d=1$, in which case it suffices to study the single scale bilinear spherical averaging operator
$$\mathcal{A}_1(f,g)(x)=\int_{S^{1}}f(x-y)g(x-z)d\sigma(y,z)=\int_{0}^{2\pi} f(x-\cos t)g(x-\sin t)dt.$$

Denote by $\mathcal{R}_1$ the boundedness region of $\mathcal A_1$, that is, 
$$\mathcal{R}_1=\{(1/p,1/q,1/r)\in[0,1]^3\colon \mathcal{A}_1:L^{p}\times L^{q}\rightarrow L^{r} \text{ is bounded} \}.$$

The region $\mathcal{R}_1$ was first investigated by Oberlin in \cite{DOberlin} in the case $p=q$. The results were further improved in \cite{BS} and then very recently in \cite{SS}, where some necessary and sufficient conditions for the $L^p$ improving estimates for $\mathcal A_1$ were obtained. However, these conditions do not yet match and one still only has a partial description of $\mathcal{R}_1$. As these conditions are fairly complicated to state, we omit them here and refer the reader to \cite[Theorem 2.1 and 2.3]{SS}.

Our main theorem of the section is the following continuity estimate.

\begin{thm}
\label{d1continuity}
Let $d = 1$, then for any $(1/p, 1/q, 1/r)$ in the interior of $\mathcal R_1$ there exists $\eta = \eta(p, q, r)>0$, such that 
\[
\|\mathcal A_1(f_1-\tau_h f_1, f_2)\|_{r}\lesssim |h|^\eta \|f_1\|_{p}\|f_2\|_{q}\]for all $|h|<1$.
\end{thm}

By a standard argument, this theorem implies the following sparse domination result for the lacunary version of the bilinear spherical maximal function
\[
\mathcal{M}_{lac}(f,g)(x)=\sup_{m\in\BZ}\mathcal{A}_{2^m}(|f|,|g|)(x).
\]

\begin{thm}\label{thm: d1}
Let $d = 1$. For all exponents $(1/p,1/q,1/r)$ in the interior of the region $\mathcal{R}_1$ such that $r>1$, the bilinear lacunary circular maximal function $\mathcal{M}_{lac}$ has $(p,q,r')$ sparse bound, where $r'$ is given by $\frac{1}{r}+\frac{1}{r'}=1$. More precisely, for all functions $f,g,h \in C_0^\infty$, there exists a sparse collection $\mathcal{S}$ of intervals in $\mathbb{R}$ such that
\[
|\langle \mathcal{M}_{lac}(f,g),h\rangle| \lesssim \sum_{Q\in \mathcal{S}} |Q| \langle f\rangle_{Q,p}\langle g\rangle_{Q,q}\langle h\rangle_{Q,r'}.
\]
\end{thm}

Theorem \ref{thm: d1} follows from Theorem \ref{d1continuity} and an almost identical (and even simpler) deduction as in Section \ref{sec: reduction} and Section \ref{sec: sparse lemma}, where the full version of the maximal function $\mathcal{M}$ in dimensions $d\geq 2$ is discussed. One needs to slightly modify those arguments because now we are in the lacunary (rather than the full) case. Such modification (for instance, see \cite{Roncal, triangle}) is routine and we omit the details. Note that an analogue of Proposition \ref{prop: double} involving translations in both components is needed as well, which is an obvious corollary of Theorem \ref{d1continuity} following the same proof of Proposition \ref{prop: double}.

There is, however, a subtlety in the proof of Theorem \ref{thm: d1} that we would like to point out. Unlike previous works, where certain versions of the assumption $p,q\leq r$ were needed in the sparse domination theorem, we are able to remove this assumption in our theorem. As detailed in the ``Bad-Good'' case in Section \ref{sec: sparse lemma} below, such assumptions are not necessary once the boundedness region for $(1/p, 1/q, 1/r)$ satisfies the following property: for all $1<p<\infty$, there exists some $q^\ast$ such that $(1/p, 1/q^\ast, 1/p)$ is in the interior of the region; symmetrically, for all $1<q<\infty$, there exists some $p^\ast$ such that $(1/p^\ast, 1/q, 1/q)$ is in the interior of the region. Even though the full description of the boundedness region $\mathcal{R}_1$ here is not yet available, it is proved in \cite[Theorem 2.3]{SS} that the region contains the convex hull of the points $(0,0,0), (0,1,1), (1,0,1), (\frac{3}{5}, \frac{3}{5}, \frac{2}{5})$, hence it has this property.

It is unclear to us whether Theorem \ref{d1continuity} can be extended to the localized bilinear circular maximal function $\sup_{1\leq t\leq 2}\mathcal{A}_t$, which, if true, would yield sparse bounds for the full version of the bilinear circular maximal function $\mathcal{M}$. 

We now move on to the proof of Theorem \ref{d1continuity}. It will be implied by the following proposition, which is in turn a consequence of the trilinear smoothing inequality (Theorem \ref{trismooth} below). We leave the proof of the proposition to the end of this section.

Let $\{Q_k\}$ be the Littlewood-Paley projections. More precisely, for any $k\in \mathbb{Z}$ define $Q_k f :=\hat \phi_k*f$ where $\phi_k = \psi_k-\psi_{k-1}$ and $\psi_k(x) = \psi(2^{-k}x)$. Here, $\psi$ is a smooth bump function supported on $[-2, 2]$, taking values in $[0, 1]$ and satisfying $\psi(\xi) \equiv 1, \forall |\xi|<1$. Let $P_k =\sum_{i\leq k} Q_i$ and $\tilde Q_i = Q_{i-1}+Q_i+Q_{i+1}$.
\begin{prop}
\label{hsproposition}
Let $d=1$. There exists $\delta>0$ such that for all functions $f_1, f_2\in L^2$ and all $k\in \mathbb{Z}$,
\[
\|\mathcal A_1(Q_kf_1, f_2)\|_{1}\lesssim 2^{-\delta k}\|f_1\|_{2}\|f_2\|_{2}.
\]
\end{prop}
\begin{proof}[Proof of Theorem \ref{d1continuity} assuming Proposition \ref{hsproposition}] We first show the continuity estimate at the tuple $(p,q,r)=(2,2,1)$. It is well known that $\mathcal{A}_1$ maps boundedly from $L^2\times L^2$ into $L^1$ (see e.g. \cite{DOberlin, SS, MCZZ}).

Fix a parameter $k\in \Z$, which shall be chosen later, and decompose $f_1$ to get 
\begin{align*}
\|\mathcal A_1(f_1 -\tau_h f_1, f_2)\|_1\lesssim &\|\mathcal A_1(P_kf_1 - \tau_h P_kf_1, f_2)\|_1\\
&+\sum_{i> k}\|\mathcal{A}_1(Q_if_1, f_2)\|_1+\sum_{i> k}\|\mathcal{A}_1(\tau_h Q_if_1, f_2)\|_1\\
&=I+II+III.
\end{align*}
The first part is bounded by $O(2^{2k}|h|\|f_1\|_2\|f_2\|_2)$ by using the $L^2\times L^2\rightarrow L^1$ boundedness of $\mathcal A_1$ and estimating $\|P_kf_1 - \tau_h P_kf_1\|_2$ via $\|\hat \psi^\prime\|_1$ and Minkowski's integral inequality. Indeed,
\begin{align*}
I& \lesssim \|P_kf_1 - \tau_h P_kf_1\|_2 \|f_2\|_2\\
&= \left\|\int_0^h f_1*\hat \psi_k^\prime(x+t)\,dt\right\|_{L^2_x} \|f_2\|_2\\
&\leq \left(\int_0^h\| f_1*\hat \psi_k^\prime(x+t)\|_{L^2_x}\,dt\right) \|f_2\|_2 \\
&\leq |h|\|f_1\|_2\|f_2\|_2\|\hat \psi_k^\prime\|_1\lesssim 2^{2k}|h||f_1\|_2\|f_2\|_2.
\end{align*}

By Proposition \ref{hsproposition}, 
\[
II = \sum_{i> k}\|\mathcal{A}_1(Q_i f_1, f_2)\|_1\leq \sum_{i> k}2^{-\delta i}\|f_1\|_2 \|f_2\|_2\lesssim_\delta 2^{-\delta k}\prod_{i = 1}^2 \|f_i\|_2.
\]Similarly, one can obtain the same estimate for the last term:
\[
III=\sum_{i> k}\|\mathcal{A}_1(\tau_h Q_if_1, f_2)\|_1=\sum_{i> k}\|\mathcal{A}_1(Q_if_1, \tau_{-h}f_2)(\cdot -h)\|_1\lesssim_\delta 2^{-\delta k}\prod_{i=1}^2\|f_i\|_2.
\]

Since we are only interested in the small $h$ case, we are allowed to choose $k$ such that $2^k = |h|^{-1/(2+\delta)}$, and we get 
$$\|\mathcal A_1(f_1 -\tau_h f_1, f_2)\|_1\leq C(\delta)|h|^{\frac{\delta}{2+\delta}}\prod_{i = 1}^2 \|f_i\|_2.$$
Interpolating this continuity estimate for $(p, q, r) = (2, 2, 1)$ with the $L^p$ improving estimate in the region $\mathcal{R}_1$ (for instance, those tuples obtained in \cite{SS}), we get continuity estimate for any point in the interior of $\mathcal R_1$. 
\end{proof}
We now prove Proposition \ref{hsproposition}, with the main tool being the trilinear smoothing inequality recalled in the theorem below. Let $\eta$ be a smooth bump function with compact support in $\mathbb R ^2$, and $\mathbf{f}=(f_0,f_1,f_2)$ be a tuple of functions on $\mathbb{R}$. Define \begin{equation}\mathcal T(\mathbf{f}) = \iint \prod_{i = 0}^2f_i\circ \phi_i(x, t)\eta(x, t) \,dxdt.\end{equation}

\begin{thm}[\cite{MCZZ}]
\label{trismooth}
Let $U$ be a connected neighborhood of the support of $\eta$.
Let $\phi_j:U\to\mathbb R^1$ be real analytic.
Assume that for any $i\ne j\in\{0,1,2\}$,
$\det(\nabla\phi_i,\nabla\phi_j)$ does not vanish identically 
in any nonempty open set.
Assume that for any nonempty connected open subset $U'\subset U$,
for any $\mathbf g\in C^\omega(\Phi(U'))$ that satisfies
$\sum_{j=0}^2 (g_j\circ\phi_j) \equiv 0$
in $U'$, each $g_j$ is constant in $\phi_j(U')$.

Then there exist $\delta>0$ and $C<\infty$
such that for all Lebesgue measurable functions 
$\mathbf{f}  = (f_0,f_1,f_2)\in L^2(\mathbb{R})\times L^2(\mathbb{R})\times L^2(\mathbb R)$,
the integral defining $\mathcal T(\mathbf{f})$ converges absolutely and 
\begin{equation}
|\mathcal T(\mathbf{f})| \leq C \prod_{j=0}^2 \|f_j\|_{H^{-\delta}}.
\end{equation}
\end{thm}
\begin{proof}[Proof of Proposition \ref{hsproposition}] 
Let $\eta$ be a smooth bump function supported on $[-\frac{1}{6}, 1+\frac{1}{6}]$, and $\eta_n(x)=\eta(x-n)$ satisfying $\sum_n\eta_n\equiv 1$. Define for each $n\in \mathbb{Z}$ that
$$\mathcal{T}_n(f_1, f_2,g) = \iint \eta_n(x)\chi_{[0, 2\pi]}(t)f_1(x+\cos t)f_2(x+\sin t)g(x)\,dxdt.$$
It is straightforward to check that $\mathcal{T}_n$ satisfies the assumption of Theorem \ref{trismooth}, and applying the trilinear smoothing inequality gives rise to the same constant for all $n$.

Then, by duality,
\[
\|\mathcal A_1(Q_kf_1,f_2)\|_1 \leq \sum_{n} \sup _{{\rm supp}g_n\subset [n-\frac{1}{6}, n+1+\frac{1}{6}],|g_n|\leq 1}|\mathcal{T}_n((Q_k f_1),f_2 ,g_n)|.
\]

We now introduce another localization to the input functions. Let $\chi$ be a smooth bump function supported on $[-\epsilon,1+\epsilon]$,
taking values in $[0,1]$, and satisfying $\chi\equiv 1$ on $[\epsilon, 1-\epsilon]$, for some $\epsilon$ sufficiently small. Denote $\chi_n(x)=\chi(x-n)$ and $\tilde \chi_n =
\chi_{n-2}+ \chi_{n-1}+\chi_{n}+\chi_{n+1}+\chi_{n+2}$. Notice that $\tilde \chi_n \equiv 1$ on $[n -2+\epsilon,n +3-\epsilon]$ and ${\rm supp} (\tilde \chi_n) \subset [n -2-\epsilon,n + 3+\epsilon]$. From the definition of $\mathcal{T}_n$, one has that the above sum is bounded by
\[
\leq\sum_n \sup_{{\rm supp}g_n\subset [n-\frac{1}{6}, n+1+\frac{1}{6}],|g_n|\leq 1}|\mathcal{T}_n((Q_kf_1)\tilde\chi_{n},f_2\tilde \chi_{n},g_n)|\leq I + II,
\]
where 
\begin{align*}
I&=\sum_n \sup_{{\rm supp}g_n\subset [n-\frac{1}{6}, n+1+\frac{1}{6}],|g_n|\leq 1}|\mathcal{T}_n(Q_k(f_1\tilde\chi_{n}),f_2\tilde \chi_{n},g_n)|,\\
II&=\sum_n \sup_{{\rm supp}g_n\subset [n-\frac{1}{6}, n+1+\frac{1}{6}],|g_n|\leq 1}|\mathcal{T}_n( (Q_kf_1)\tilde\chi_{n}- Q_k(f_1\tilde\chi_{n}) ,f_2\tilde \chi_{n},g_n)|.
\end{align*}

Notice that $\|g_n\|_{H^{-\delta}} \leq \|g_n\|_2 \leq 1$ uniformly over $n$. By the trilinear smoothing inequality and Cauchy-Schwarz, 
$$ I\lesssim \sum_n\|Q_k(\tilde \chi_n f_1)\|_{H^{-\delta}} \|\tilde \chi_n f_2\|_{H^{-\delta}}\leq 2^{-\delta k}(\sum_n \|\tilde \chi_n f_1\|_2^2)^{1/2}(\sum_n \|\tilde \chi_n f_2\|_2^2)^{1/2}\lesssim 2^{-\delta k}\|f_1\|_2\|f_2\|_2.$$

For the error part $II$, let $m_k(x) = 2^k(1+2^k|x|^2)^{-10}$, which satisfies $|\hat{\phi}_k(x)|\chi_{|x|\geq 1}\lesssim 2^{-k}m_k(x)$. 
The following pointwise estimate holds true by the fast decaying property of $\hat\phi$ and the fact that $\tilde \chi_n(x)-\tilde \chi_n(u)\equiv 0$ when $|x-u|<1, x\in {\rm supp}(\eta_n)$:
\begin{align*}
|(Q_kf_1)\tilde\chi_{n}- Q_k(f_1\tilde\chi_{n})(x)| &= \left|\int_{|x-u|>1} (\tilde \chi_n(x)-\tilde \chi_n(u))f_1(u)\hat\phi_k(x-u)\,du\right|\\
&\lesssim 2^{-k}\int |f_1(u)|\,m_k(x-u)\,du \\
&=2^{-k}|f_1|*m_k(x). \end{align*}
This then implies
\begin{align*}
II &\lesssim \sum_n \sup_{{\rm supp}g_n\subset [n-\frac{1}{6}, n+1+\frac{1}{6}],|g_n|\leq 1} \langle  \mathcal{A}_1(2^{-k}|f_1|*m_k, |f_2|)\eta_n, |g_n|\rangle\\
&\leq \sum_n\|\mathcal A_1(2^{-k}|f_1|*m_k, |f_2|)\eta_n\|_1\\
&\lesssim 2^{-k}\||f_1|*m_k\|_2\|f_2\|_2\lesssim 2^{-k}\|f_1\|_2\|f_2\|_2.
\end{align*}
where we used that $\mathcal{A}_1:L^2\times L^2\rightarrow L^1$ is bounded, which is an easy consequence of Holder's inequality.
\end{proof}

The proof above works for averaging operators over a more general class of curves in $\mathbb R^2$ as well. See \cite{MCZZ} for necessary conditions on this class of curves.

\section{Reduction to dyadic maximal operators}\label{sec: reduction}

In this section, we follow closely the argument in \cite{triangle} to reduce the problem to sparse bounds for dyadic maximal operators. The arguments in this section work for all $d\geq 1$.

Without loss of generality, assume $f,g\geq 0$. If $\mathcal{D}'$ is a dyadic lattice whose cubes have side lengths in the set $\{\frac{1}{3}2^j\},\,j\in \Z$, then by the Three Lattice Theorem 
$$\{3Q\colon Q\in \mathcal{D}'\}=\cup_{i=1}^{3^d}\mathcal{D}^i,$$
where each $\mathcal{D}^i$ is dyadic lattice whose cubes have side lengths in the set $\{2^j\colon j\in \Z\}$. Denote 
\begin{align*}
\mathcal{D}_m'&=\left\{Q\in \mathcal{D}'\colon l(Q)=\frac{1}{3}2^{m}\right\}, \\
\mathcal{D}_m^{i}&=\left\{Q\in \mathcal{D}^i\colon l(Q)=2^{m}\right\}.
\end{align*}

Define 
$$\mathcal{A}_{*,t}(f,g)(x):=\sup_{s\in [t,2t]} \mathcal{A}_s(f,g)(x).$$

Then
\begin{equation*}
    \begin{split}
        \mathcal{A}_{*,2^{m-4}}(f,g)(x)=&\mathcal{A}_{*,2^{m-4}}(\sum_{Q\in \mathcal{D}'_m}f\chi_{Q},\sum_{Q'\in \mathcal{D}'_m}g\chi_{Q'})(x)\\
        \leq &\sum_{Q,Q'\in \mathcal{D}'_m} \mathcal{A}_{*,2^{m-4}}(f\chi_Q,g\chi_{Q'})(x).
    \end{split}
\end{equation*}

We claim that if $\mathcal{A}_{*,2^{m-4}}(f\chi_Q,g\chi_{Q'})(x)\neq 0$, then $Q'\subseteq 3Q$. Indeed, if 
$$\sup_{t\in[2^{m-4},2^{m-3}]}\int_{S^{2d-1}}(f\chi_Q)(x-ty)(g\chi_{Q'})(x-tz)d\sigma (y,z)\neq 0,$$
take $t\in [2^{m-4},2^{m-3}]$ and $(y,z)\in S^{2d-1}$ such that $x-ty\in Q$ and $x-tz\in Q'$ so 
$$d(Q,Q')\leq |(x-ty)-(x-tz)|\leq t|y-z|\leq 2^{m-2}<\frac{1}{3}2^m=l(Q)\Rightarrow Q'\subseteq 3Q.$$

Let $Q(1),Q(2),\dots,Q(3^d)$ be an enumeration of the cubes in $\mathcal{D}_m'$ that comprise $3Q$. Then 
\begin{equation*}
    \begin{split}
        \mathcal{A}_{*,2^{m-4}}(f,g)(x)\leq &\sum_{Q\in \mathcal{D}'_m}\sum_{j=1}^{3^d} \mathcal{A}_{*,2^{m-4}}(f\chi_Q,g\chi_{Q(j)})(x)\\
        =:&\sum_{Q\in \mathcal{D}'_m}\sum_{j=1}^{3^d} \mathcal{A}_{*,Q,Q(j)}(f,g)(x).
    \end{split}
\end{equation*}

Moreover, if $x\in \text{supp}\,( \mathcal{A}_{*,Q,Q(j)}(f,g) )$, then $(x-ty)\in Q$ for some $(y,z)\in S^{2d-1}$ and $t\in [2^{m-4},2^{m-3}]$, so $d(x,Q)\leq |ty|\leq 2^{m-3}|y|<\frac{1}{3}2^m $. This implies that
$$\text{supp} \mathcal{A}_{*,Q,Q(j)}(f,g)\subseteq 3Q.$$

Putting all of this together,
\begin{equation*}
    \begin{split}
        \mathcal{M}(f,g)(x)=&\sup_{m\in\Z}\mathcal{A}_{*,2^{m-4}}(f,g)(x)\\
        \leq &\sup_{m\in \Z} \sum_{Q\in \mathcal{D}_m'}\sum_{j=1}^{3^d} \mathcal{A}_{*,Q,Q(j)}(f,g)(x)\\
        =&\sup_{m\in \Z} \sum_{i=1}^{3^d} \sum_{Q\in \mathcal{D}^{i}_m}\sum_{j=1}^{3^d} \mathcal{A}_{*,\frac{1}{3}Q,(\frac{1}{3}Q)(j)}(f,g)(x)\\
        =&:\sup_{m\in \Z} \sum_{i,j=1}^{3^d} \sum_{Q\in \mathcal{D}^{i}_m} \mathcal{A}_{*,Q}^j(f,g)(x) \quad\text{ (with ${\rm supp}(\mathcal{A}_{*,Q}^j)(f,g)\subseteq Q$)}\\
        \leq & \sum_{i,j=1}^{3^d}\sup_{m\in \Z} \sup_{Q\in \mathcal{D}^{i}_m} \mathcal{A}_{*,Q}^j(f,g)(x)\quad\text{ (because the cubes in $\mathcal{D}_m^{i}$ are disjoint)}\\
        =& \sum_{i,j=1}^{3^d}\sup_{Q\in \mathcal{D}^{i}} \mathcal{A}_{*,Q}^j(f,g)(x)=:\sum_{i,j=1}^{3^d} \mathcal{M}_{*,j}^{i}(f,g)(x).
    \end{split}
\end{equation*}

Hence, it is enough to find sparse bounds for dyadic maximal operators of the form 
$$\mathcal{M}_{*,j}(f,g)(x)=\sup_{Q\in \mathcal{D}}\mathcal{A}_{*,Q}^j(f,g)(x)=\sup_{Q\in \mathcal{D}}\mathcal{A}_{*,2^{m-4}}(f\chi_{\frac{1}{3}Q},g\chi_{(\frac{1}{3}Q)(j)})(x),$$
where $\mathcal{D}$ is a dyadic lattice, and for each $Q$, $m$ is given by $l(Q)=2^m$.

The previous reduction to the sparse domination of dyadic maximal operators, combined with the following lemma, will allow us to reduce our problem to prove the analogue of Lemma 4.2 in \cite{triangle}, which we will state and prove in Section \ref{sec: sparse lemma}. 

\begin{lem}\label{localization}
Let $\mathcal{D}$ be a dyadic lattice. Suppose $f,g,h$ are bounded compactly supported functions in $\tilde{Q}\in \mathcal{D}$. Then if $Q_0\in \mathcal{D}$ is a large enough ancestor, one has
$$\langle \mathcal{A}_{*,Q'}^j(f,g) ,h\rangle=0, \text{ for all }Q'\in \mathcal{D},\, Q_0\subset Q'.$$
\end{lem}

\begin{proof}
Since $\text{supp}(h)\subset \tilde{Q}$, it is enough to show that $\mathcal{A}_{*,Q'}^j(f,g)(x)=0$ for all $x\in \tilde{Q}$, if $Q'$ is a large enough ancestor of $\tilde{Q}$.

Suppose there exists $x\in \tilde{Q}$ such that $\mathcal{A}_{*,Q'}^j(f,g)(x)\neq 0$. Then for $2^m=l(Q')$, one has
\begin{equation*}
    0\neq\sup_{2^{m-4}\leq t\leq 2^{m-3}} \int_{S^{2d-1}}(f\chi_{\frac{1}{3}Q'})(x-ty)(g\chi_{(\frac{1}{3}Q')(j)})(x-tz)\,d\sigma(y,z),
\end{equation*}which implies that there is $t\in [2^{m-4}, 2^{m-3}]$, $(y,z)\in S^{2d-1}$ such that
\[
(f\chi_{\frac{1}{3}Q'})(x-ty)(g\chi_{(\frac{1}{3}Q')(j)})(x-tz)\neq 0.
\]
This then yields that $( x-ty,x-tz)\in \text{supp}(f)\times\text{supp}(g)\subseteq  \tilde{Q}\times \tilde{Q}$. Since $(y,z)\in S^{2d-1}$, we have $|y|\geq \frac{1}{\sqrt{2}}$ or $|z|\geq \frac{1}{\sqrt{2}}$. Say $|y|\geq \frac{1}{\sqrt{2}}$, then 

\begin{equation*}
    \begin{split}
       \frac{l(Q')}{2^{9/2}}\leq 2^{m-4}|y|\leq  |(x-ty)-x|\leq \text{diam}(\tilde{Q}),
    \end{split}
\end{equation*}
which will not be satisfied for $Q'$ sufficiently large.

\end{proof}

\section{Proof of Theorem \ref{thm: main}}\label{sec: sparse lemma}
We follow the exposition in \cite{Roncal} and \cite{triangle} most closely in this section. 

By arguments already contained in \cite{triangle} we can reduce the proof of the sparse bounds to the case when $f$ and $g$ are characteristic functions and $h\geq 0$ is a compactly supported bounded function.

\begin{lem}\label{indicator}
Let $d\geq 2$. Let $(1/p,1/q,1/r)$ with $r>1$ in the boundedness region $\mathcal{R}(d)$. Then for compactly supported indicator functions $f=\chi_F$ and $g=\chi_G$, and $h\geq 0$ compactly supported bounded function, there exists a sparse collection $\mathcal{S}$ such that 
$$\langle \mathcal{M}(f,g),h \rangle \lesssim \sum_{Q\in \mathcal{S}} |Q|\langle f\rangle_{Q,p}\langle g\rangle_{Q,q}\langle h\rangle_{Q,r'}.$$
\end{lem}

Lemma \ref{indicator} follows from the main lemma below and the two facts that we observed in Section \ref{sec: reduction}. One of them was that we could reduce the proof of the sparse domination to sparse bounds for dyadic maximal operators of the form
$$\mathcal{M}_{*,j}(f,g)(x)=\sup_{Q\in \mathcal{D}}\mathcal{A}_{*,2^{m-4}}(f\chi_{\frac{1}{3}Q},g\chi_{(\frac{1}{3}Q)(j)})=\sup_{Q\in \mathcal{D}}\mathcal{A}_{*,Q}^j(f,g)(x),$$
where $\mathcal{D}$ is a dyadic lattice, and for each $Q$, $m$ is given by $l(Q)=2^m$. The second ingredient is the localization property proved in Lemma \ref{localization}, which allows us for given functions $f,g,h$ to pass from sums over all dyadic cubes to sums of subcubes of some dyadic cube $Q_0$. The details can be found
in \cite{triangle}, when they show that Lemma 4.2 implies Lemma 4.1, with small modifications because we are in the full (rather than lacunary) case. We omit the details.

Our main Theorem \ref{thm: main} will thus follow from the next lemma. The continuity estimates we developed in Section \ref{sec: continuity} will play a key role in its proof. To simplify the notation, $\langle f\rangle_{Q,p}$ below will denote the $p$ average of $|f|$ over the cube $Q$.

\begin{lem}\label{lem: key}
Let $d\geq 2$ and $(1/p,1/q,1/r)\in \mathcal{R}(d)$, the interior of the known boundedness region of  $\tilde{\mathcal{M}}$ given in \eqref{eq: bdd region}, satisfying $r>1$. Let $f=\chi_{F}$, $g=\chi_{G}$ be measurable indicator functions supported in a dyadic cube $Q_0$ and $h$ a bounded measurable function also supported in $Q_0$. Suppose there exists $C_0>1$ and $\mathcal{D}_0$ be a collection of dyadic subcubes of $Q_0$ such that 
\begin{equation*}
    \sup_{Q'\in \mathcal{D}_0}\, \sup_{Q\colon Q'\subseteq Q\subseteq Q_0}\left(\dfrac{\langle f\rangle_{Q,p}}{\langle f\rangle_{Q_0,p}}+\dfrac{\langle g\rangle_{Q,q}}{\langle g\rangle_{Q_0,q}}+\dfrac{\langle h\rangle_{Q,r'}}{\langle h\rangle_{Q_0,r'}}\right)\leq C_0.
\end{equation*}
Then if $\{B_Q\}$ is a family of pairwise disjoint sets indexed by $Q\in \mathcal{D}_0$ with $B_Q\subset Q$ and $h_Q:=h\chi_{B_Q}$, we have the estimate 
$$\sum_{Q\in \mathcal{D}_0} \langle \mathcal{A}_{*,Q}^j(f,g),h_Q \rangle\lesssim |Q_0|\langle f \rangle_{Q_0,p}\langle g\rangle_{Q_0,q} \langle h \rangle_{Q_0,r'} $$
where $\mathcal{A}_{*,Q}^j(f,g)(x)=\mathcal{A}_{*,2^{m-4}}(f\chi_{\frac{1}{3}Q},g \chi_{(\frac{1}{3} Q)(j)})(x)$, for $l(Q)=2^m$.

\end{lem}

\begin{proof}

Without loss of generality, assume $h\geq 0$. It will be convenient to work with the slightly larger operators. If $l(Q)=2^m$,
\begin{equation*}
    \begin{split}
       \mathcal{A}_{*,Q}^j(f,g)(x)=&\sup_{t\in [2^{m-4},2^{m-3}]}\mathcal{A}_t(f\chi_{\frac{1}{3}Q},g\chi_{(\frac{1}{3}Q)(j)})(x)\\
       \leq & \sup_{t\in [2^{m-4},2^{m-3}]}\mathcal{A}_t(f\chi_{\frac{1}{2}Q},g\chi_{\tilde{Q}(j)})(x)=:\mathcal{S}_{*,Q}^j(f,g)(x)
    \end{split}
\end{equation*}
where $\tilde{Q}(j)$ is the union of dyadic children of $Q$ which cover the cube $(\frac{1}{3}Q)(j)$.

The operator $\mathcal{S}_{*,Q}^j(f,g)$ is bi-sublinear but we can define bilinear operators to control it. Namely for a given measurable function $\kappa:\R^d\rightarrow [1,2]$ define
$$\mathcal{S}_{\kappa,Q}^j(f,g)(x)=\mathcal{A}_{2^{m-4}\kappa(x)}(f\chi_{\frac{1}{2}Q},g\chi_{\tilde{Q}(j)})(x).$$

It is enough to prove that for all measurable $\kappa:\R^d\rightarrow [1,2]$, one has
$$\left|\sum_{Q\in \mathcal{D}_0} \langle \mathcal{S}_{\kappa,Q}^j(f,g),h_Q \rangle\right|\lesssim |Q_0|\langle f \rangle_{Q_0,p}\langle g\rangle_{Q_0,q} \langle h \rangle_{Q_0,r'}. $$

Following \cite{triangle}, one can define adjoints for $\mathcal{A}_{2^{m-4}\kappa}$. Recall that since $(1/p,1/q,1/r)\in \mathcal{R}(d)$,  
$$|\langle \mathcal{A}_{2^{m-4}\kappa}(f,g),h \rangle |\leq \|\mathcal{A}_{2^{m-4}\kappa}(f,g)\|_r\|h\|_{r'}\lesssim (2^m)^{d(\frac{1}{r}-\frac{1}{p}-\frac{1}{q})}\|f\|_p\|g\|_q\|h\|_{r'}.$$

This implies that for any $g\in L^q$, and $h\in L^{r'}$, the map $f \mapsto \langle \mathcal{A}_{2^{m-4}\kappa}(f,g),h \rangle$ is a bounded linear functional on $L^p$. By $L^p$ duality there exists $\mathcal{A}_{2^{m-4}\kappa}^{*,1}(g,h)\in L^{p'}$ such that 
$$\langle \mathcal{A}_{2^{m-4}\kappa}(f,g),h \rangle=\langle f, \mathcal{A}_{2^{m-4}\kappa}^{*,1}(g,h) \rangle, \quad \forall f\in L^p. $$

Similarly one can define $\mathcal{A}_{2^{m-4}\kappa}^{*,2}(f,h)\in L^{q'}$ such that 
$$\langle \mathcal{A}_{2^{m-4}\kappa}(f,g),h \rangle=\langle g, \mathcal{A}_{2^{m-4}\kappa}^{*,2}(f,h) \rangle ,\quad \forall g\in L^q.$$

Now, we perform a Calder\'on-Zygmund decomposition on $f,g$. Write $f=\gamma_{f}+\beta_{f}$, where the selected cubes satisfy
\begin{equation}
    B_f=\left\{\text{maximal dyadic cubes }P\subset Q_0: \frac{\langle f\rangle_{P,p}}{\langle f\rangle_{Q_0,p}}> C_0\right\},
\end{equation}
and as usual, we set
\begin{equation}
    \beta_f=\sum_{P\in B_f}(f-\langle f\rangle_{P})\chi_P.
\end{equation}

Similarly, $g=\gamma_g+\beta_{g}$.
Moreover, we can group all the cubes in $B_f$ of side length $2^k$ together, call this collection $B_{f,k}$, and let $\beta_{f,k}$ denote the part of $\beta_f$ supported on those cubes. Thus, we have
\begin{equation}
    \beta_f=\sum_{k=-\infty}^{m_0-1}\beta_{f,k}
\end{equation}
where $m_0=\log_2(l(Q_0))$.

As usual, the maximality of the selected cubes guarantees that the good part of $f$ (the $\gamma_f$) satisfies nice estimates such as $\| \gamma_f\|_{L^{\infty}}\lesssim \langle f\rangle_{Q_0,p}$. We break into a fourfold good-bad decomposition:
\begin{equation}
\begin{split}
    |\sum_{Q\in\mathcal{D}_0}\langle \mathcal{A}_{*,Q}^j(f,g),h_Q\rangle| &\le |\sum_{Q\in\mathcal{D}_0}\langle \mathcal{A}_{*,Q}^j(\gamma_f,\gamma_g),h_Q\rangle|+|\sum_{Q\in\mathcal{D}_0}\langle \mathcal{A}_{*,Q}^j(\beta_f,\gamma_g),h_Q\rangle|\\
    &+|\sum_{Q\in\mathcal{D}_0}\langle \mathcal{A}_{*,Q}^j(\gamma_f,\beta_g),h_Q\rangle|+|\sum_{Q\in\mathcal{D}_0}\langle \mathcal{A}_{*,Q}^j(\beta_f,\beta_g),h_Q\rangle| \\
    &=:GG+BG+GB+BB.
    \end{split}
\end{equation}

\textbf{The GG term}: This term can be dealt with easily using H\"older's inequality and the disjointness of the sets $B_Q$:
\begin{equation}
\begin{split}
    GG&\le \sum_{Q\in\mathcal{D}_0} \| \gamma_f\|_{\infty}\|\gamma_g\|_{\infty}\| h_Q\|_1 \\
    &\lesssim \langle \gamma_f\rangle_{Q_0, p}\langle \gamma_g\rangle_{Q_0, q}|Q_0|\langle h \rangle_{Q_0,1} \\
    &\le |Q_0|\langle \gamma_f\rangle_{Q_0, p}\langle \gamma_g\rangle_{Q_0, q}\langle h\rangle_{Q_0,r'}.
    \end{split}
\end{equation}
\textbf{The BG term}: 

\begin{equation*}
    \begin{split}
      BG=&  |\sum_{Q\in\mathcal{D}_0}\langle \mathcal{A}_{*,Q}^j(\beta_f,\gamma_g),h_Q\rangle|\\
      \leq &\sum_{q\in \mathcal{D}_0}\sum_{k=1}^{3} |\langle \mathcal{A}_{*,Q}^j(\beta_{f,m-k},\gamma_g),h_Q\rangle|+\sum_{q\in \mathcal{D}_0}\sum_{k=4}^{\infty} |\langle \mathcal{A}_{*,Q}^j(\beta_{f,m-k},\gamma_g),h_Q\rangle|.
    \end{split}
\end{equation*}

For $1\leq k\leq 3$, just using the boundedness properties of $\tilde{\mathcal{M}}$ we get
\begin{equation*}
    \begin{split}
       &\sum_{Q\in\mathcal{D}_0}|\langle \mathcal{A}_{*,Q}^j(\beta_{f,m-k},\gamma_g),h_Q\rangle|\\
       \lesssim &  \sum_{Q\in \mathcal{D}_0} 2^{md(\frac{1}{r_0}-\frac1p-\frac{1}{q_0})} \| \beta_{f,m-k}\|_{L^{p}(Q)} \| \gamma_{g}\|_{L^{q_0}(Q)} \| h_Q\|_{L^{r_0'}(Q)}\\
       \lesssim & \sum_{Q\in \mathcal{D}_0} |Q| \langle \beta_{f,m-k}\rangle_{Q,p} \langle \gamma_{g}\rangle_{Q,q_0} \langle h_Q\rangle_{Q,r_0'}  
    \end{split}
\end{equation*}
as long as $(1/p,1/q_0,1/r_0)\in \mathcal{R}(d)$.

Mimicking the argument in \cite[Section 4.3]{triangle} we get to the point where for all $Q\in \mathcal{D}_0$ with $l(Q)=2^m$, and for all $k\geq 4$,

\begin{equation}\label{bgtriangle}
    \begin{split}
        |\langle S_{\kappa,Q}^j&(\beta_{f,m-k},\gamma_g), h_Q\rangle |\\
        \lesssim & \dfrac{1}{|P_0|}\int_{P_0} \int_{\R^d} \mathcal{A}_{2^{m-4}\kappa}\left([I-\tau_{-y}]\left(\text{sign}(I_1(\cdot,\cdot-y))\beta_{f,m-k}\chi_{\frac{1}{2}Q}\right),\gamma_g \chi_{\tilde{Q}(j)}\right)(x)\\
        &\qquad\qquad\qquad\qquad\qquad\qquad\qquad\cdot h_Q(x)\,dx dy,
    \end{split}
\end{equation}
where $P_0$ is the cube centered at the origin with sidelength $2l(P)=2\cdot 2^{m-k}$, and
\[
\begin{split}
&I_1(x,x')\\
:=&\beta_{f,m-k}(x)\chi_{\frac{1}{2}Q}(x)\left\{ \mathcal{A}_{2^{m-4}\kappa}^{*,1}(\gamma_g \chi_{\tilde{Q}(j)} ,h_Q)(x)-\mathcal{A}_{2^{m-4}\kappa}^{*,1}(\gamma_g \chi_{\tilde{Q}(j)} ,h_Q)(x')\right\}.
\end{split}
\]

Remember that the condition $k\geq 4$ was important to guarantee that if $P\in B_{f,m-k}$ with $P\cap \frac{1}{2}Q\neq \emptyset$, then $P\subset \frac{1}{2} Q$, and analogously for $\tilde{Q}(j) $. Then we can exploit the fact that $\beta_f$ has average $0$ in each of the bad cubes $P$. The main reason why we passed to the slightly larger cubes $\frac{1}{2} Q$ and $\tilde{Q}(j)$ is to have these nesting properties satisfied.

Once we have (\ref{bgtriangle}), we may use the continuity estimates that we proved in Proposition \ref{prop: cont} for $(1/p,1/q_0,1/r_0)\in \mathcal{R}(d)$ to be chosen smartly later. The right hand side of (\ref{bgtriangle}) is then bounded by
\begin{equation*}
    \begin{split}
        \lesssim & \dfrac{1}{|P_0|}\int_{P_0}\left(\dfrac{|y|}{2^m}\right)^{\eta}2^{md(1/r_0-1/p-1/q_0)}\|\beta_{f,m-k}\|_{L^{p}(Q)}\|\gamma_{g}\|_{L^{q_0}(Q)}\|h_Q\|_{L^{r_0'}(Q)}\, dy\\
        \lesssim & 2^{-k\eta}2^{md} (2^{-md/p}\|\beta_{f,m-k}\|_{L^{p}(Q)})(2^{-md/q_0}\|\gamma_{g}\|_{L^{q_0}(Q)})(2^{-md/r_0'}\|h_Q\|_{L^{r_0'}(Q)}) \\
        = & 2^{-k\eta}|Q| \langle \beta_{f,m-k}\rangle_{Q,p} \langle \gamma_{g}\rangle_{Q,q_0} \langle h_Q\rangle_{Q,r_0'}. \\
    \end{split}
\end{equation*}

Observe that independently of the choice of $q_0\geq 1$, one always has
$$\langle\gamma_{g}\rangle_{Q,q_0}\leq \|\gamma_g\|_{\infty}\lesssim \langle g\rangle_{Q_0,q}.$$

Now all we need to check is that for a suitable choice of $(\frac{1}{q_0},\frac{1}{r_0})$, with $(\frac{1}{p},\frac{1}{q_0},\frac{1}{r_0})\in \mathcal{R}(d)$ one has for all $k\geq 1$ that
\begin{equation}\label{bgestimate}
    \sum_{Q\in \mathcal{D}_0}|Q|\langle\beta_{f,m-k} \rangle_{Q,p} \langle h_Q\rangle_{Q,r_0'}\lesssim |Q_0| \langle f\rangle_{Q_0,p} \langle h\rangle_{Q_0,r'}.
\end{equation}

For completeness, we state and prove the following lemma, which was already implicit in \cite{Lacey} and \cite{Roncal}.

\begin{lem}\label{lemmalacey}
Let $k\geq 1$ fixed. Let $p_1,p_2\in [1,\infty)$ such that $1/p_1+1/p_2\geq 1$. Then 
$$\sum_{Q\in \mathcal{D}_0}|Q|\langle\beta_{f,m-k} \rangle_{Q,p_1} \langle h_Q\rangle_{Q,p_2}\lesssim |Q_0|\langle f\rangle_{Q_0,p_1} \langle h\rangle_{Q_0,p_2},$$
with implicit constant independent of $k$.
\end{lem}

Before proving it, let us see how Lemma \ref{lemmalacey} implies (\ref{bgestimate}).

We will consider two cases, $1/p\geq 1/r$ and $1/p< 1/r$, and we describe the choice of the auxiliary pair $(1/q_0,1/r_0)$ for each of those cases. Basically all we need is a triple $(1/p,1/q_0,1/r_0)\in \mathcal{R}(d)$ with $1/p+1/r_0'\geq 1$, so we can apply Lemma \ref{lemmalacey}.

If $1/p\geq 1/r$, one can simply take $(1/q_0,1/r_0)=(1/q,1/r)$. Then Lemma \ref{lemmalacey} can be applied directly to get (\ref{bgestimate}), because $1/p+1/r'\geq 1$.

If $1/p<1/r$, take $r_0=p$ and choose $1<q_0<\infty$ large enough so that $1/p+1/q_0< m(d,p)$, where $m(d,p)$ is the quantity defining the region $\mathcal{R}(d)$ for the continuity estimates. Then, the continuity estimate holds for $(1/p,1/q_0,1/r_0)$ and again since  $1/p+1/r_0'=1/p+1/p'=1$ we can get from Lemma \ref{lemmalacey} that 
$$\sum_{Q\in \mathcal{D}_0}|Q|\langle\beta_{f,m-k} \rangle_{Q,p} \langle h_Q\rangle_{Q,r_0'}\lesssim |Q_0| \langle f\rangle_{Q_0,p} \langle h\rangle_{Q_0,r_0'} .$$

\begin{figure}[h]
\begin{center}
         \scalebox{0.9}{
\begin{tikzpicture}

\fill[blue!10!white] (0,2.5)--(2.5,0)--(3,0)--(3,1.7)--(1.7,3)--(0,3)--(0,2.5);

\draw (-0.3,4.1) node {$y$};
\draw (4.1,-0.1) node {$x$};
\draw (3,-0.4) node {$1$};
\draw (-0.4,3) node {$1$};
\draw (2.5,-0.4) node {$\frac{1}{p}$};
\draw (-0.4,2.5) node {$1/p$};

\draw[->,line width=1pt] (-0.2,0)--(4,0);
\draw[->,line width=1pt] (0,-0.2)--(0,4);

\draw[-,line width=1pt] (3,-0.1)--(3,0.1);
\draw[-,line width=1pt] (-0.1,3)--(0.1,3);
\draw[-,line width=1pt] (-0.1,2.5)--(0.1,2.5);
\draw[-,line width=1pt] (2.5,-0.1)--(2.5,0.1);

\draw[-,line width=1pt, red] (2.5,0)--(2.5,3);
\filldraw[black] (0,0) circle (1.5pt);

\filldraw[red] (2.5,2.5) circle (1.5pt);
\filldraw[red] (2.5,1) circle (1.5pt);

\draw[red] (3.15,2.5) node {$(\frac{1}{p},\frac{1}{q})$};
\draw[red] (3.1,1) node {$(\frac{1}{p},\frac{1}{q_0})$};
\draw[dashed,blue,line width=1pt] (0,2.5)--(2.5,0)--(3,0)--(3,1.7)--(1.7,3)--(0,3)--(0,2.5);

\draw[green!70!black] (3,1.7)--(1.7,3);
\draw[green!70!black] (4.5,1.7) node {$x+y=m(d,p)$};

\end{tikzpicture}}
\end{center}
\caption{Figure that illustrates how for each $p$ one can find $q_0$ such that $1/p< 1/p+1/q_0<m(d,p)$, that is, $(1/p,1/q_0,1/p)\in \mathcal{R}(d)$.}
\end{figure}
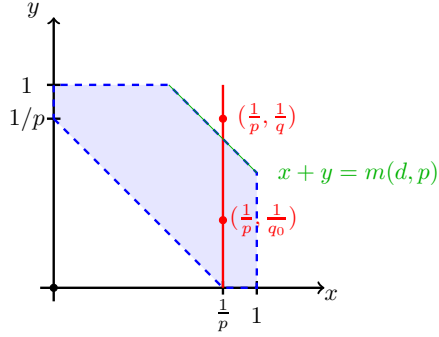

Notice that $1/p<1/r$ implies $1/p'>1/r'$ and thus $r_0'=p'<r'$. Hence, we can bound $\langle h\rangle_{Q_0,r_0'}\leq \langle h\rangle_{Q_0,r'}$, which implies (\ref{bgestimate}).

\begin{rem}
Notice that exploiting the fact that $\gamma_g$ is bounded we were able to choose the triple $(1/p,1/q_0,1/r_0)$ with $r_0\geq r$ and $q_0$ conveniently, so we can handle both cases $1/p<1/r$ and $1/p\geq 1/r$. This is an improvement in comparison  with the estimate of the BG part in
\cite{triangle} where they assumed $1/p\geq 1/r$. Similarly we don't need to assume $1/q\geq 1/r$ for the estimate of the GB term.
\end{rem}

\begin{proof}[Proof of Lemma \ref{lemmalacey}]
This proof depends on the fact that $f=\chi_F$ and $g=\chi_G$ are characteristic functions.  
\begin{equation*}
    \begin{split}
        \beta_{f,m-k}=&\sum_{P\in B_{f,m-k}}(f-\langle f\rangle_P )\chi_P=\sum_{P\in B_{f,m-k}}\chi_{P\cap F}-\sum_{P\in B_{f,m-k}}\langle f\rangle_P \chi_P\\
        =&\chi_{F_{m,k}} -\sum_{P\in B_{f,m-k}}\langle f\rangle_P \chi_P,
    \end{split}
\end{equation*}where $F_{m,k}=\sqcup_{P\in B_{f,m-k}} (P\cap F)$. Therefore,
\[
 |\beta_{f,m-k}|\leq  \chi_{F,m,k}  +\langle f \rangle_{Q_0}  \chi_{E_{f,m,k}},
\]where $E_{f,m,k}=\sqcup_{P\in B_{f,m-k}}P$. Here, we used that by the maximality of the bad cubes of $f$, $\langle f \rangle_{P}\lesssim \langle f\rangle_{Q_0}$.

We can then write 

\begin{equation*}
    \begin{split}
        \sum_{Q\in \mathcal{D}_0}|Q|\langle\beta_{f,m-k} \rangle_{Q,p_1} \langle h_Q\rangle_{Q,p_2}\lesssim & \sum_{Q\in \mathcal{D}_0}|Q|\langle \chi_{F_{m,k}} \rangle_{Q,p_1} \langle h_Q\rangle_{Q,p_2}\\
        +&\langle f \rangle_{Q_0,p_1}\sum_{Q\in \mathcal{D}_0}|Q|\langle \chi_{E_{f,m,k}} \rangle_{Q,p_1} \langle h_Q\rangle_{Q,p_2}=:I+II.
    \end{split}
\end{equation*}

Since $1/p_1+1/p_2\geq 1$, let $\tau\geq 0$ such that $1/p_1+1/p_2=1+\tau$, define $\dot{p_1}$ via $\frac{1}{\dot{p_1}}=\frac{1}{p_1}-\tau$. Then since $Q\in \mathcal{D}_0$,
\begin{equation*}
\begin{split}
     \langle \chi_{F_{m,k}} \rangle_{Q,p_1}=&\langle \chi_{F_{m,k}}\cdot \chi_{F_{m,k}} \rangle_{Q,p_1}\leq \langle \chi_{F_{m,k}} \rangle_{Q,\dot{p_1}}\langle \chi_{F_{m,k}} \rangle_{Q,1/\tau}\\
    \leq &\langle \chi_{F_{m,k}} \rangle_{Q,\dot{p_1}}\langle \chi_{F} \rangle_{Q}^{\tau}\lesssim \langle \chi_{F_{m,k}} \rangle_{Q,\dot{p_1}}\langle \chi_{F} \rangle_{Q_0}^{\tau}.
\end{split}
\end{equation*}

Now we use this and the fact that $1/\dot{p_1}+1/p_2=1$ to get
\begin{equation*}
    \begin{split}
        I=&\sum_{Q\in \mathcal{D}_0}|Q|\langle \chi_{F_{m,k}} \rangle_{Q,p_1} \langle h_Q\rangle_{Q,p_2}\leq \langle \chi_{F} \rangle_{Q_0}^{\tau} \sum_{Q\in \mathcal{D}_0}|Q|\langle \chi_{F_{m,k}} \rangle_{Q,\dot{p_1}} \langle h_Q\rangle_{Q,p_2}\\
        =&\langle \chi_{F} \rangle_{Q_0}^{\tau} \sum_{Q\in \mathcal{D}_0}\left(\int_Q \chi_{F_{m,k}}\right)^{1/\dot{p_1}}  \left(\int h_Q^{p_2}\right)^{1/p_2}\\
        \leq &\langle \chi_{F} \rangle_{Q_0}^{\tau} \left(\sum_{m}\sum_{Q\in \mathcal{D}_0\colon l(Q)=2^m}  \int_Q \chi_{F_{m,k}}\right)^{1/\dot{p_1}}  \left(\sum_{Q\in \mathcal{D}_0} \int_{B_Q} h^{p_2}\right)^{1/p_2}\\
        \leq &\langle \chi_{F} \rangle_{Q_0}^{\tau} \left(\sum_{m}  \int_{Q_0} \chi_{F_{m,k}}\right)^{1/\dot{p_1}}  \left( \int_{Q_0} h^{p_2}\right)^{1/p_2}\\
        \leq& |Q_0|  \left(\frac{|F|}{|Q_0|}\right)^{1/\dot{p_1}+\tau} \langle h \rangle_{Q_0,p_2}=|Q_0|  \langle f \rangle_{Q_0,p_1} \langle h \rangle_{Q_0,p_2}.
    \end{split}
\end{equation*}

 For $II$ the estimate is simpler:
\begin{equation*}
    \begin{split}
    \sum_{Q\in \mathcal{D}_0} |Q| \langle \chi_{E_{f,m,k}}\rangle_{Q,p_1} \langle h_Q\rangle_{Q,p_2}&\leq \sum_{Q\in \mathcal{D}_0} |Q| \langle \chi_{E_{f,m,k}}\rangle_{Q,\dot{p_1}}\langle h_Q\rangle_{Q,p_2}\\
    &\leq \sum_{Q\in \mathcal{D}_0} \left(\int_{Q} \chi_{E_{f,m,k}}\right)^{1/\dot{p_1}}\left(\int h_Q^{p_2}\right)^{1/p_2}\\
    &\leq \left(\sum_{Q\in \mathcal{D}_0}\int_{Q} \chi_{E_{f,m,k}}\right)^{1/\dot{p_1}} \left(\sum_{Q\in \mathcal{D}_0}\int h^{p_2}\chi_{B_Q}\right)^{1/p_2}\\
    &\leq |Q_0|^{1/\dot{p_1}}\left(\int_{Q_0}h^{p_2}\right)^{1/p_2}= |Q_0|\langle h\rangle_{Q_0,p_2}.
    \end{split}
\end{equation*}
\end{proof}

\noindent The \textbf{GB} term is treated symmetrically to the \textbf{BG} term.\\

\noindent \textbf{The BB term}: 

By similar computations as in \cite[Section 4.3]{triangle} we can use the boundedness properties of $\tilde{\mathcal{M}}$, the continuity estimates proved in Section \ref{sec: continuity}, and the adjoint operators $\mathcal{A}_{2^{m-4}\kappa}^{*,1}$ and $ \mathcal{A}_{2^{m-4}\kappa}^{*,2}$ to get the estimate

\begin{equation}
    \begin{split}
        BB\leq & \sup_{\kappa}\sum_{Q\in \mathcal{D}_0}|\langle S_{\kappa,Q}^j(\beta_{f},\beta_g), h_Q\rangle |\leq \sup_{\kappa}\sum_{Q\in \mathcal{D}_0} \sum_{k=1}^{\infty}\sum_{j=1}^{\infty} |\langle S_{\kappa,Q}^j(\beta_{f,m-k},\beta_{g,m-j}), h_Q\rangle|
        \\ \lesssim &\sum_{k=1}^{\infty} \sum_{j=1}^{\infty} 2^{-\eta(j+k)} \sum_{Q\in \mathcal{D}_0} |Q|\langle\beta_{f,m-k}\rangle_{Q,p}\langle \beta_{g,m-j} \rangle_{Q,q} \langle h_Q\rangle_{Q,r'}.
    \end{split}
\end{equation}
Lemma \ref{roncalsimplified} below will then imply that

$$BB\lesssim \sum_{j=1}^{\infty} 2^{-\eta(j+k)} |Q_0|\langle f\rangle_{Q_0,p}\langle g \rangle_{Q_0,q}\langle h\rangle_{Q_0,r'}\lesssim |Q_0|\langle f\rangle_{Q_0,p}\langle g \rangle_{Q_0,q}\langle h\rangle_{Q_0,r'}$$
which concludes the proof of Lemma \ref{lem: key}.
\end{proof}

\begin{lem}\label{roncalsimplified}
 Let $f=\chi_F$ and $g=\chi_G$ as before, and $1<r<\infty$. Assume $1/p+1/q+1/r'\geq 1$, then 
$$\sum_{Q\in \mathcal{D}_0} |Q|\langle\beta_{f,m-k}\rangle_{Q,p}\langle \beta_{g,m-j} \rangle_{Q,q} \langle h_Q\rangle_{Q,r'}\lesssim |Q_0|\langle f\rangle_{Q_0,p}\langle g \rangle_{Q_0,q}\langle h\rangle_{Q_0,r'}$$
with implicit constant independent of $k$ and $j$.
\end{lem}

\begin{proof}[Proof of Lemma \ref{roncalsimplified}]
We start with the bounds 
\begin{equation*}
    \begin{split}
        \langle\beta_{f,m-k} \rangle_{Q,p} &\lesssim \langle \chi_{F,m,k}\rangle_{Q,p} +\langle f \rangle_{Q_0,p} \langle\chi_{E_1,m,k}\rangle_{Q,p}\\
         \langle\beta_{g,m-j} \rangle_{Q,q} &\lesssim \langle \chi_{G,m,j}\rangle_{Q,q} +\langle g \rangle_{Q_0,q} \langle\chi_{E_2,m,k}\rangle_{Q,q},
    \end{split}
\end{equation*}
where 
\begin{equation*}
    \begin{split}
        E_{1,m,k}=&\bigsqcup_{P\in B_{f,m-k}} P\,,\hspace{1cm}
        F_{m,k}=\bigsqcup_{P\in B_{f,m-k}} P\cap F;\\
        E_{2,m,j}=&\bigsqcup_{P\in B_{g,m-j}} P\,,\hspace{1cm}
        G_{m,j} =\bigsqcup_{P\in B_{g,m-j}} P\cap G.
    \end{split}
\end{equation*}
Therefore 
\begin{equation*}
    \begin{split}
        &\sum_{Q\in \mathcal{D}_0} |Q|\langle\beta_{f,m-k}\rangle_{Q,p}\langle \beta_{g,m-j} \rangle_{Q_0,q} \langle h_Q\rangle_{Q,r'}\\
        \lesssim & 
        \sum_{Q\in \mathcal{D}_0} |Q|\langle \chi_{F,m,k}\rangle_{Q,p}\langle \chi_{G,m,j} \rangle_{Q,q} \langle h_Q\rangle_{Q,r'}\\
        &+\langle g\rangle_{Q_0,q} \sum_{Q\in \mathcal{D}_0} |Q|\langle \chi_{F,m,k}\rangle_{Q,p}\langle \chi_{E_{2,m,j}} \rangle_{Q,q} \langle h_Q\rangle_{Q,r'}\\
        &+\langle f\rangle_{Q_0,p} \sum_{Q\in \mathcal{D}_0} |Q|\langle \chi_{E_{1,m,k}} \rangle_{Q,p} \langle \chi_{G,m,j}\rangle_{Q,q}\langle h_Q\rangle_{Q,r'}\\
        &+\langle f\rangle_{Q_0,p} \langle g\rangle_{Q_0,q} \sum_{Q\in \mathcal{D}_0} |Q|\langle \chi_{E_1,m,k}\rangle_{Q,p}\langle \chi_{E_{2,m,j}} \rangle_{Q_0,q} \langle h_Q\rangle_{Q,r'}\\
        =:&BB_1+BB_2+BB_3+BB_4.
    \end{split}
\end{equation*}

Since $1/p+1/q+ 1/r'\geq 1$, take $0\leq\tau_1 <1/p$ and $0\leq \tau_2 < 1/q$ such that $\frac{1}{p}+\frac{1}{q}+\frac{1}{r'}=1+\tau_1+\tau_2$. Define $\dot{p},\dot{q}$ by $\frac{1}{\dot{p}}=\frac{1}{p}-\tau_1$ and $\frac{1}{\dot{q}}=\frac{1}{q}-\tau_2$, then the H\"older relation $\frac{1}{\dot{p}}+\frac{1}{\dot{q}}+\frac{1}{r'}=1$ is satisfied.

\begin{equation*}
    \begin{split}
        BB_1=&\sum_{Q\in \mathcal{D}_0} |Q|\langle \chi_{F,m,k}\rangle_{Q,p}\langle \chi_{G,m,j} \rangle_{Q,q} \langle h_Q\rangle_{Q,r'}\\
        \leq &\langle\chi_F \rangle_{Q_0}^{\tau_1}\langle\chi_G \rangle_{Q_0}^{\tau_2}\sum_{Q\in \mathcal{D}_0} (|Q|^{1/\dot{p}}\langle \chi_{F,m,k} \rangle_{Q,\dot{p}})(|Q|^{1/\dot{q}}\langle \chi_{G,m,j} \rangle_{Q,\dot{q}})(|Q|^{1/r'}\langle h_Q \rangle_{Q,r'})\\
        \leq &\langle\chi_F \rangle_{Q_0}^{\tau_1}\langle\chi_G \rangle_{Q_0}^{\tau_2}\left(\sum_{Q\in \mathcal{D}_0} \int_Q \chi_{F,m,k}\right)^{1/\dot{p}}\left(\sum_{Q\in \mathcal{D}_0} \int_Q \chi_{G,m,j}\right)^{1/\dot{q}}\left(\sum_{Q\in \mathcal{D}_0} \int_Q h_Q\right)^{1/\dot{r'}}\\
        \leq & \left(\frac{|F|}{|Q_0|}\right)^{\tau_1+1/ \dot{p}}\left(\frac{|G|}{|Q_0|}\right)^{\tau_2+1/\dot{q}}|Q_0|^{\frac{1}{\dot{p}}+\frac{1}{\dot{q}}}\left(\int_{Q_0} h^{r'}\right)^{1/r'}\\
=&|Q_0|\langle f\rangle_{Q_0,p}\langle g \rangle_{Q_0,q}\langle h\rangle_{Q_0,r'}.        
    \end{split}
\end{equation*}

Notice that 
$$\frac{1}{\dot{q}}=\frac{1}{q}-\tau_2\leq \frac{1}{q} \implies \dot{q}\geq q\implies \langle \chi_{E_{2,m,j}} \rangle_{Q,q}\leq  \langle \chi_{E_{2,m,j}} \rangle_{Q,\dot{q}}.  $$This allows us to estimate $BB_2$ in the following way:
\begin{equation*}
    \begin{split}
        BB_2=&\langle g\rangle_{Q_0,q}\sum_{Q\in \mathcal{D}_0} |Q|\langle \chi_{F,m,k}\rangle_{Q,p}\langle \chi_{E_{2,m,j}} \rangle_{Q,q} \langle h_Q\rangle_{Q,r'}\\
        \leq & \langle g\rangle_{Q_0,q}\langle\chi_F \rangle_{Q_0}^{\tau_1}\sum_{Q\in \mathcal{D}_0} |Q|\langle \chi_{F,m,k} \rangle_{Q,\dot{p}} \langle \chi_{E_{2,m,j}} \rangle_{Q,\dot{q}}\langle h_Q \rangle_{Q,r'}\\
        \leq & \langle g\rangle_{Q_0,q}\langle\chi_F \rangle_{Q_0}^{\tau_1} |F|^{1/\dot{p}}|Q_0|^{1/\dot{q}} \left(\int_{Q_0} h^{r'}\right)^{1/r'}\\
=&|Q_0|\langle f\rangle_{Q_0,p}\langle g \rangle_{Q_0,q}\langle h\rangle_{Q_0,r'}.        
    \end{split}
\end{equation*}The term $BB_3$ is symmetric to $BB_2$, and finally we have 
\begin{equation*}
    \begin{split}
        BB_4=&\langle f\rangle_{Q_0,p}\langle g\rangle_{Q_0,q}\sum_{Q\in \mathcal{D}_0} |Q|\langle \chi_{E_{1,m,k}}\rangle_{Q,p}\langle \chi_{E_{2,m,j}} \rangle_{Q,q} \langle h_Q\rangle_{Q,r'}\\
        \leq&\langle f\rangle_{Q_0,p}\langle g\rangle_{Q_0,q}\sum_{Q\in \mathcal{D}_0} |Q|\langle \chi_{E_{1,m,k}}\rangle_{Q,\dot{p}}\langle \chi_{E_{2,m,j}} \rangle_{Q,\dot{q}} \langle h_Q\rangle_{Q,r'}\\
      \leq& \langle f\rangle_{Q_0,p}\langle g\rangle_{Q_0,q} |Q_0|^{1/\dot{p}}|Q_0|^{1/\dot{q}} \left(\int_{Q_0} h^{r'}\right)^{1/r'} \\  
=&|Q_0|\langle f\rangle_{Q_0,p}\langle g \rangle_{Q_0,q}\langle h\rangle_{Q_0,r'}.        
    \end{split}
\end{equation*}

This finishes the proof of Lemma \ref{roncalsimplified}
\end{proof}

\section{Sharpness of the Bounds}\label{sec: sharp}
By considering specific examples, we can show sharpness of part of the range of exponents obtained for the sparse bounds in Theorem \ref{thm: main} and the continuity estimates for the localized bilinear spherical maximal function $\tilde{\mathcal{M}}$ in Proposition \ref{prop: cont}, up to the boundary. We rescale to unit size for simplicity in all cases. We will reuse some examples considered in \cite[Proposition 3.3]{JL}, and we also introduce a third Knapp-type example (motivated by an example in \cite{Lacey}) which extends the necessary conditions derived in \cite[Proposition 3.3]{JL}.

\subsection{Sharpness of the range for sparse bounds}
The following proposition establishes necessary conditions on the range of exponents $(p, q,r)$ for which sparse bound of the bilinear spherical maximal function holds.

\begin{prop}
Let $d\geq 2$, $1<p,q<\infty$, and $1<r\leq \infty$. Suppose further that $\frac{1}{p}+\frac{1}{q}>\frac{1}{r}$ and the bilinear spherical maximal function $\mathcal{M}$ has $(p,q,r')$ sparse bound. Then,
\[
\frac{1}{p}+\frac{1}{q}\leq\min\left\{1+\frac{d}{r}, \frac{2d-1}{d},\frac{2d}{d+1}+\frac{d-1}{r(d+1)}\right\}.
\]

\end{prop}

Note that sparse bound is only interesting when $\frac{1}{p}+\frac{1}{q}>\frac{1}{r}$, as otherwise one can trivially get $(p,q,r')$ sparse bound using the $L^p\times L^q \to L^{r_0}$ bound of $\mathcal{M}$, for $1/r_0:=1/p+1/q$, and H\"older's inequality. 

Comparing the necessary conditions in the above to the sufficient conditions for sparse bounds stated in Theorem \ref{thm: main}, one sees that the first two necessary conditions are also sufficient (up to the boundary), while the third one is a bit weaker than the sufficient condition $\frac{1}{p}+\frac{1}{q}< \frac{1}{r}+\frac{2(d-1)}{d}$. In fact, as pointed out in \cite{JL}, it is also unknown whether this last sufficient condition is necessary for the $L^p$ improving estimate for the localized bilinear spherical maximal function $\tilde{\mathcal{M}}$. An immediate observation from the proof below is that the third example we present below also shows that if $\tilde{\mathcal{M}}$ is bounded from $L^p\times L^q$ to $L^r$, then $\frac{1}{p}+\frac{1}{q}\leq \frac{2d}{d+1}+\frac{d-1}{r(d+1)} $. This improves the best known unboundedness range for $L^p$ improving estimate for $\tilde{\mathcal{M}}$ \cite[Proposition 3.3]{JL}.

\begin{proof}
 We present three examples here, each of which will imply one of the claimed upper bounds. 

First, fix $0<\epsilon_0\ll 1$ sufficiently small. For $0<\delta\leq\epsilon_0$, consider functions $f_{1,\delta}=\chi_{B^d(0,\delta)}$, $g_{1,\delta}=\chi_{B^d(0,C\delta)}$, and $h_1=\chi_{\mathbb{A}}$, where $\mathbb{A}=\{x\in \mathbb{R}^d:\, \frac{1}{\sqrt{2}}\leq |x|\leq \frac{1}{\sqrt{2}}+\epsilon_0\}$. 
It is proved in \cite[Proposition 3.3]{JL} that there exists a choice of constant $C>0$ such that
\[
\tilde{\mathcal{M}}(f_{1,\delta},g_{1,\delta})(x)\gtrsim \delta^{2d-1},\quad \forall x\in \mathbb{A}.
\]
Indeed, $C=100$ is certainly sufficient for our purposes and it works in \cite{JL} argument, as can be seen by writing the operator in sliced form and considering intersections of balls and spheres centered at points in $\mathbb{A}$ with points in $B^d(0,\delta)$. This immediately implies the lower bound 
\[
|\langle \mathcal{M}(f_{1,\delta},g_{1,\delta}),h_1 \rangle|\geq |\langle \tilde{\mathcal{M}}(f_{1,\delta},g_{1,\delta}),h_1 \rangle|\gtrsim \delta^{2d-1}|\mathbb{A}|.
\]

On the other hand, in any sparse form $\sum_{Q\in \mathcal{S}}|Q|\langle f_{1,\delta}\rangle_{Q,p}\langle g_{1,\delta}\rangle_{Q,q}\langle h_1\rangle_{Q,r'}$, if $Q\in \mathcal{S}$ makes a nonzero contribution, then $Q$ needs to intersect the support of all three functions. Thus, $Q$ will have scale at least $1/2$ and the contribution of all such $Q$s decreases as the side length of the cube increases. Therefore, it suffices to consider the case that $\mathcal{S}$ contains a single cube $Q$ that has side length $\sim 1$. One thus has the upper bound
\[
\sum_{Q\in \mathcal{S}}|Q|\langle f_{1,\delta}\rangle_{Q,p}\langle g_{1,\delta}\rangle_{Q,q}\langle h_1\rangle_{Q,r'}\lesssim \|f_{1,\delta}\|_{L^p}\|g_{1,\delta}\|_{L^q}\|h_1\|_{L^{r'}}\sim \delta^{d(\frac{1}{p}+\frac{1}{q})}|\mathbb{A}|^{1/r'}.
\]Playing these two bounds against each other, one immediately has $\frac{1}{p}+\frac{1}{q}\leq \frac{2d-1}{d}$.

The second example is very similar, which we also borrow from \cite[Proposition 3.3]{JL}. Let $\delta>0$ be sufficiently small. Consider
\[
f_{2,\delta}=\chi_{B^d(0,\frac{1}{\sqrt{2}}+2\delta)\setminus B^d(0, \frac{1}{\sqrt{2}}-2\delta)},\quad g_{2,\delta}=\chi_{B^d(0,\frac{1}{\sqrt{2}}+C\delta)\setminus B^d(0, \frac{1}{\sqrt{2}}-C\delta)},
\]and $h_{2,\delta}=\chi_{B^d(0,\delta)}$. It is shown in \cite[Proposition 3.3]{JL} that there exists a choice of $C$ (which can be taken to be 100, similarly to the previous example) such that for all $|x|\leq \delta$, $\tilde{\mathcal{M}}(f_{2,\delta},g_{2,\delta})(x)\gtrsim \delta$. Therefore, one has
\[
|\langle \mathcal{M}(f_{2,\delta},g_{2,\delta}),h_{2,\delta} \rangle|\geq |\langle \tilde{\mathcal{M}}(f_{2,\delta},g_{2,\delta}),h_{2,\delta} \rangle|\gtrsim \delta^{1+d}.
\]Similarly as in the first example, one also has the upper bound
\[
\sum_{Q\in \mathcal{S}}|Q|\langle f_{2,\delta}\rangle_{Q,p}\langle g_{2,\delta}\rangle_{Q,q}\langle h_{2,\delta}\rangle_{Q,r'}\lesssim \|f_{2,\delta}\|_{L^p}\|g_{2,\delta}\|_{L^q}\|h_{2,\delta}\|_{L^{r'}}\sim \delta^{\frac{1}{p}+\frac{1}{q}+\frac{d}{r'}},
\]which implies $\frac{1}{p}+\frac{1}{q}\leq 1+\frac{d}{r}$.

The last example is adapted from \cite[Proposition 5.1]{Lacey} and is a Knapp type example. Let $\delta>0$. Consider three rectangles in $\mathbb{R}^d$: \[
R_1=[-C_1\sqrt{\delta}, C_1\sqrt{\delta}]^{d-1}\times [-C_1\delta, C_1\delta],\quad  R_2=[-C_2\sqrt{\delta}, C_2\sqrt{\delta}]^{d-1}\times [-C_2\delta, C_2\delta],
\]
\[R_3=[-\sqrt{\delta},\sqrt{\delta}]^{d-1}\times [\frac{1}{\sqrt{2}},\sqrt{2}],
\]where the constants $C_1, C_2$ can be chosen to be 100 if $\delta$ is sufficiently small. Let $f_{3,\delta}=\chi_{R_1}$, $g_{3,\delta}=\chi_{R_2}$, and $h_{3,\delta}=\chi_{R_3}$. For every $x\in R_3$, one has from the slicing formula \cite[(2.2)]{JL} that
\[
\begin{split}
\tilde{\mathcal{M}}(f_{3,\delta},g_{3,\delta})(x)\gtrsim & \sup_{1\leq t\leq 2}\int_{\frac{1}{\sqrt{2}}-\delta<|y|<\frac{1}{\sqrt{2}}} f_{3,\delta}(x-ty)\, \int_{S^{d-1}}g_{3,\delta}(x-t\sqrt{1-|y|^2}z)\,d\sigma(z)dy\\
\gtrsim & \int_{\frac{1}{\sqrt{2}}-\delta<|y|<\frac{1}{\sqrt{2}}}\chi_{R_1}(x-t(x)y)\delta^{\frac{d-1}{2}}\,dy.
\end{split}
\]
The last step above follows from the fact that for all $x\in R_3$, there is always some $t=t(x)=\sqrt{2}x_d\in [1,2]$ such that for all $\frac{1}{\sqrt{2}}-\delta<|y|<\frac{1}{\sqrt{2}}$, the circle centered at $x$ of radius $t|y|$ intersects a large portion of $R_1$ (with arc length comparable to $\delta^{(d-1)/2}$), and the circle centered at $x$ of radius $t\sqrt{1-|y|^2}$ intersects a large portion of $R_2$ (again with arc length comparable to $\delta^{(d-1)/2}$). This further implies that the above is bounded from below by
\[
\gtrsim \delta^{\frac{d-1}{2}}\delta \delta^{\frac{d-1}{2}}=\delta^d.
\]This implies the lower bound
\[
|\langle \mathcal{M}(f_{3,\delta},g_{3,\delta}),h_{3,\delta} \rangle|\gtrsim \delta^d |R_3|=\delta^{d+\frac{d-1}{2}}. 
\]On the other hand, by a similar argument as in the previous examples, one can obtain an upper bound
\[
\sum_{Q\in \mathcal{S}}|Q|\langle f_{3,\delta}\rangle_{Q,p}\langle g_{3,\delta}\rangle_{Q,q}\langle h_{3,\delta}\rangle_{Q,r'}\lesssim \|f_{3,\delta}\|_{L^p}\|g_{3,\delta}\|_{L^q}\|h_{3,\delta}\|_{L^{r'}}\sim \delta^{\frac{d+1}{2}(\frac{1}{p}+\frac{1}{q})+\frac{d-1}{2r'}}.
\]The sparse bound then implies that $\frac{1}{p}+\frac{1}{q}\leq \frac{2d}{d+1}+\frac{d-1}{r(d+1)}$.
\end{proof}

\subsection{Sharpness of the range for continuity estimates}

We now discuss sharpness of the range of exponents in Proposition \ref{prop: cont} for the continuity estimates of $\tilde{\mathcal{M}}$. The first necessary condition obtained below is based on standard results regarding multilinear Fourier multipliers. We include the argument here for the sake of completeness. In fact, it suffices to prove this for the single-scale spherical average $\mathcal{A}_1$, which is pointwisely bounded by $\tilde{\mathcal{M}}$ and thus the necessary condition extends to $\tilde{\mathcal{M}}$ as well.

\begin{prop}
Suppose that we have for some $0<p,q,r<\infty$ the following estimate
\[
\|\mathcal{A}_1(f, g-\tau_h g)\|_{L^r}\leq C |h|^\eta \|f\|_{L^p}\|g\|_{L^q}
\]uniform across all $|h|\le 1$. Then $\frac{1}{p}+\frac{1}{q}\ge\frac{1}{r}$.
\end{prop}
\begin{proof}
For simplicity, denote $\mathcal{A}_h(f,g)(x)= \mathcal{A}_1(f, g-\tau_h g)(x)$. This operator commutes with simultaneous translations, which is to say that for any vector $v$, we have the identity
\begin{equation}
    \tau_v\mathcal{A}_h(f,g)(x)=\mathcal{A}_h(\tau_vf,\tau_vg)(x).
\end{equation}
We can fix a value of $h>0$ and $f,g$ such that equality holds in this bound. By standard density arguments, we may assume that $f,g$ both have compact support as well. We then have that
\begin{equation}
    \|\mathcal{A}_h(f+\tau_vf,g+\tau_vg)\|_{L^r}\le C|h|^\eta\|f+\tau_vf\|_{L^p}\|g+\tau_v g\|_{L^q}.
\end{equation}
We also have by bilinearity and translation invariance that the left hand side of the inequality can be rewritten as
\begin{equation}
    \|\mathcal{A}_h(f+\tau_vf,g+\tau_vg)\|_{L^r}=\|\mathcal{A}_h(f,g)+\tau_v\mathcal{A}_h(f,g)+\mathcal{A}_h(\tau_vf,g)+\mathcal{A}_h(f,\tau_vg)\|_{L^r}.
\end{equation}
Since $f,g$ have compact support, if $|v|$ is sufficiently large then the last two terms are identically 0. Finally, recalling the fact that for $s<\infty$
\begin{equation}
    \|k+\tau_vk\|_{L^s}\rightarrow 2^{1/s}\|k\|_{L^s} \text{ as }|v|\rightarrow \infty,
\end{equation}
we can send $|v|$ to infinity to get the estimate
\begin{equation}
    2^{1/r}\|\mathcal{A}_h(f,g)\|_{L^r} \leq C |h|^\eta 2^{\frac{1}{p}+\frac{1}{q}}\|f\|_{L^p}\|g\|_{L^q}.
\end{equation}
Since $f,g,h$ were chosen to make the original estimate hold with equality, we get a contradiction unless $\frac{1}{p}+\frac{1}{q}\ge\frac{1}{r}$.
\end{proof}

When $r>1$ the necessary conditions on the exponents for sparse domination $(p,q,r')$ are also necessary conditions for the continuity estimate (since when $r>1$ a larger region of continuity estimates would imply a larger range where the sparse bounds hold). If $r\leq 1$, one can observe that 
$$\min\left\{1+\frac{d}{r}, \frac{2d-1}{d},\frac{2d}{d+1}+\frac{d-1}{r(d+1)}\right\}=\frac{2d-1}{d}.$$
so all we need to show is the necessity of the condition $1/p+1/q\leq \frac{2d-1}{d}$. This can be done by using the corresponding example in the previous subsection and by taking the translation parameter $h$ in the continuity estimate such that $|h|=\delta^{1/4}$ for instance with $h$ aligned in the $x_1$-direction.
We summarize this observation in the following corollary and omit the proof; see \cite[Proposition 5.5]{Lacey} for a similar reasoning in the linear case.

\begin{cor}
Let $d\geq 2$. Suppose that we have for some $p,q,r$ the following estimate
\begin{equation}
\|\tilde{\mathcal{M}}(f, g-\tau_h g)\|_{L^r}\leq C |h|^\eta \|f\|_{L^p}\|g\|_{L^q}
\end{equation}
uniform across all $|h|\le 1$. Then 
\begin{equation}
\frac{1}{p}+\frac{1}{q}\le\min\left\{1+\frac{d}{r}, \frac{2d-1}{d},\frac{2d}{d+1}+\frac{d-1}{r(d+1)}\right\}.
\end{equation}
\end{cor}

\bibliographystyle{alpha}
\bibliography{sources}

\end{document}